\documentclass[,11pt,hyp,]{nyjm}
\usepackage{hyperref}
\hypersetup{nesting=true,debug=true,naturalnames=true}
\newcommand{\rmref}[1]{{\rm\ref{#1}}}
\usepackage{graphicx,amssymb}
\let\<\langle
\let\>\rangle
\newcommand{\doi}[1]{doi:\,\href{http://dx.doi.org/#1}{#1}}

\renewcommand{\theenumi}{\roman{enumi}}

\title{Isospectral metrics on weighted projective spaces}

\author[Martin Weilandt]{\hyperlink{Weilandt}{Martin Weilandt}}
\address{\hypertarget{Weilandt}Humboldt-Universit\"at zu Berlin, Institut f\"ur Mathematik, Unter den Linden 6, 10099 Berlin, Germany}
\curraddr{Departamento de Matem\'atica, Universidade Federal de Santa Catarina, Campus Universit\'ario Trindade, Florian\'opolis/SC, 88040-900, Brazil}
\email{martin@mtm.ufsc.br}
\thanks{This research was funded by the Berlin Mathematical School and partially funded by the DFG Sonderforschungsbereich SFB 647}



\newtheorem{theorem}{Theorem}[section]    
\newtheorem*{theorem*}{Theorem}  

\newtheorem{lemma}[theorem]{Lemma}          
\newtheorem{corollary}[theorem]{Corollary}

\newtheorem{proposition}[theorem]{Proposition}

\theoremstyle{definition}
\newtheorem{definition}[theorem]{Definition}    
\newtheorem*{remark}{Remark}             

\newtheorem{notations}[theorem]{Notations and Remarks}

\newcommand{\C}{\mathbb{C}}
\newcommand{\Q}{\mathbb{Q}}
\newcommand{\R}{\mathbb{R}}
\newcommand{\Z}{\mathbb{Z}}
\newcommand{\co}{\colon\thinspace}
\newcommand{\CP}{\mathbb{CP}}

\renewcommand{\O}{\mathcal{O}}
\newcommand{\ol}[1]{\overline{#1}}
\DeclareMathOperator{\re}{Re}
\DeclareMathOperator{\tr}{tr}

\DeclareMathOperator{\spec}{spec}
\DeclareMathOperator{\grad}{grad}
\DeclareMathOperator{\Isom}{Isom}
\DeclareMathOperator{\id}{id}

\DeclareMathOperator{\dvol}{dvol}
\DeclareMathOperator{\Aut}{Aut}
\DeclareMathOperator{\Diffeo}{Diffeo}
\DeclareMathOperator{\spann}{span}
\DeclareMathOperator{\Id}{Id}

\newcommand{\qak}[1]{{\left[#1\right]}} 
\newcommand{\qao}[1]{{\left[#1\right]}} 

\newcommand{\qbo}[1]{{\ol{#1}}} 

\newcommand*{\reg}{\text{reg}} 

\numberwithin{equation}{section}

\begin{document}

\begin{abstract}    
We construct the first examples of families of bad Riemannian orbifolds which are isospectral with respect to the Laplacian but not 
isometric. In our case these are particular fixed weighted projective spaces equipped with isospectral metrics obtained by a 
generalization of Sch\"uth's version of the torus method.
\end{abstract}
\maketitle
\tableofcontents


\section{Introduction}
An orbifold is a generalization of a smooth manifold which is in general not locally homeomorphic to an open subset of $\R^n$ but to 
the quotient of a smooth manifold $\widetilde{U}$ by an effective action of a finite group $\Gamma$. A Riemannian metric is then in 
each orbifold chart as above given by a $\Gamma$-invariant metric on $\widetilde{U}$. Given a Riemannian metric on an orbifold, it 
is possible to generalize the manifold Laplacian and it is well-known that in the compact setting the spectrum of the orbifold 
Laplacian can be written as an infinite sequence
\[0=\lambda_0\le\lambda_1\le\lambda_2\le\lambda_3\le\cdots\nearrow \infty\]
of eigenvalues, each repeated according to the (finite) dimension of the corresponding eigenspace (\cite{MR2433665}). The 
observation that the spectrum contains geometric information like dimension, volume and certain curvature integrals gave rise to the 
field of spectral geometry which asks about the degree to which the spectrum of the Laplacian determines the geometry of the given 
space.

Besides a vast theory on manifolds (cf.~\cite{MR1736857}), the spectral geometry on orbifolds has recently received rising 
attention, since these provide the arguably simplest type of singular spaces, and it is still an open problem, whether a singular 
orbifold can be isospectral to (i.e., have the same spectrum as) a manifold (though there are some results on isotropy groups, which 
in a way measure the degree of singularity of an orbifold (\cite{MR2447904,MR2263484,MR2155380}). However, all known nontrivial 
examples of isospectral orbifolds (also compare \cite{MR1152950,MR1181812,MR2761691,MR2802590,sutton10}) are good, i.e., they can be 
written as the quotient of a Riemannian manifold $M$ by a discrete subgroup $\Gamma$ of the isometry group of $M$, and the 
eigenspaces on the orbifold $M/\Gamma$ correspond to the $\Gamma$-invariant eigenspaces on $M$. Since the known constructions can be 
seen to never yield an isospectral pair of a manifold and a singular orbifold, the 
more intricate setting of bad (i.e., non-good) orbifolds deserves special attention. The isospectrality of bad orbifolds was already 
investigated in \cite{MR2395193} and \cite{MR2529473}, where large families of nonhomeomorphic weighted projective spaces with their 
standard metrics were shown to be pairwise nonisospectral. Weighted projective spaces are a generalization of complex projective 
space obtained by taking certain quotients of odd-dimensional spheres by $S^1$-actions with finite stabilizers.

In this work we now use certain weighted projective spaces and special metrics based on ideas in \cite{MR1895349} to construct isospectral metrics on bad orbifolds. Our main result is the following Theorem \ref{theorem:maintheorem}.

\begin{theorem*}
 For every $n\ge 4$ and for all pairs $(p,q)$ of coprime positive integers there are isospectral families of pairwise nonisometric metrics on the orbifold $\O(p,q)$, a weighted projected space of dimension $2n\ge 8$, which is a bad orbifold for $(p,q)\ne (1,1)$.
\end{theorem*}

This theorem generalizes a result on $\CP^n$ (which is the case $(p,q)=(1,1)$ in the theorem above) from \cite{rueckriemen06}.

This paper is organized as follows: Section \ref{section:preliminaries} summarizes basic notions on orbifolds and some facts from the spectral geometry of compact Riemannian orbifolds.

In Section \ref{section:torus-method} we generalize results from \cite{MR1895349} to orbifolds. The basic idea of this so-called 
torus method (which, in a different form, was first used in \cite{MR1298201}) is that the existence of related isometric actions of 
a fixed torus on two Riemannian orbifolds implies under very special conditions that these two orbifolds are isospectral. We also 
point out how the criterion for nonisometry in \cite{MR1895349} generalizes to the orbifold case.

In Section \ref{section:examples} we introduce (with $n,p,q$ as in the theorem above) our weighted projective spaces 
$\O(p,q)=S^{2n+1}/S^1$ with the action given by $\sigma(u,v)=(\sigma^pu,\sigma^q v)$ for $\sigma\in S^1\subset\C$, $u\in\C^{n-1}$, 
$v\in\C^2$. We then fix $n,p,q$ and use the results from Section \ref{section:torus-method} together with other ideas from 
\cite{MR1895349} to obtain families of isospectral metrics on the orbifold $\O(p,q)$. For the impatient reader 
Section~\ref{subsubsection:examples-families} contains an alternative isospectrality proof independent of 
Section~\ref{section:torus-method} which also implies our main result but applies only to the case of isospectral families and hence 
misses some potential isospectral pairs. Eventually, in Section \ref{subsection:nonisometry} we show that the resulting metrics are 
(under certain conditions) nonisometric, thus establishing our main theorem above. Moreover, inspired by \cite{sutton10}, we give 
isospectral metrics on quotients of our weighted projective spaces by certain finite groups in Section \ref{subsection:sutton}.

\subsubsection*{Acknowledgements} This work is a condensed version of my Ph.D. thesis at the Humboldt-Universit\"at zu Berlin and I 
am indebted to my supervisor Dorothee Sch\"uth. Without her foresight this project never would have come into being and without her 
unceasing guidance and curiosity it could not have been finished. I would also like to thank Emily Dryden, Alexander Engel, Luis 
Guijarro and the referees for helpful suggestions.

\section{Orbifold preliminaries}
\label{section:preliminaries}
The concept of an orbifold was introduced by Satake in \cite{MR0079769} and popularized by Thurston (\cite{thurston}). We basically 
follow Satake's definition, also compare \cite{MR1950941, MR2012261, weilandt07} for basic introductions to orbifolds. However, 
since there seems to be no standard reference for orbifolds from the point of view of differential geometry, we will summarize basic 
notions and results which are necessary for the constructions in the following sections. For an extended version of this section 
with detailed proofs see \cite{weilandt10}. 

\subsection{Basics}
\label{subsection:basics}
An $n$-dimensional orbifold \emph{chart} on a topological space $X$ is given by a tuple $(U,\widetilde{U}/\Gamma,\pi)$ where $U$ is 
an open connected subset of the \emph{underlying space} $X$, $\widetilde{U}$ is a connected $n$-dimensional smooth manifold and 
$\Gamma$ is a finite group acting smoothly and effectively on $\widetilde{U}$. $\pi$ is a continuous map $\widetilde{U}\to U$ which 
induces a homeomorphism $\widetilde{U}/\Gamma\to U$. Two $n$-dimensional charts $(U_i,\widetilde{U}_i/\Gamma_i,\pi_i)$, $i=1,2$ on 
the same space $X$ are called compatible if for every $x\in U_1\cap U_2$ there is an $n$-dimensional orbifold chart 
$(U,\widetilde{U}/\Gamma,\pi)$ on $X$ such that $x\in U\subset U_1\cap U_2$ and there are smooth embeddings (so-called 
\emph{injections}) $\lambda_1\co \widetilde{U}\to \widetilde{U}_1$, $\lambda_2\co \widetilde{U}\to \widetilde{U}_2$ satisfying 
$\pi_1\circ\lambda_1=\pi=\pi_2\circ\lambda_2$. A covering of $X$ by compatible charts is called an orbifold 
\emph{atlas}. An \emph{orbifold} is then a pair $\O=(X,\mathfrak{A})$ of a second-countable Hausdorff space $X$ and a maximal atlas 
$\mathfrak{A}$ on $X$. If $\O$ is connected, the dimension of $\O$ is by definition given by the dimension of the manifolds 
$\widetilde{U}$ appearing in charts on $\O$.

The \emph{isotropy} of a point $x\in\O$ is the isomorphism class of the stabilizer $\Gamma_{\tilde{x}}$, where 
$(U,\widetilde{U}/\Gamma,\pi)$ is an arbitrary chart around $x$ and $\tilde{x}\in\pi^{-1}(x)$. It is not hard to show that the 
compatibility conditions above imply that the isotropy is well-defined. Points with trivial isotropy are called \emph{regular}, 
nonregular points are called \emph{singular}. It is well-known that in every connected orbifold $\O$ the set of regular points 
(which we will denote by $\O^\reg$) is a connected manifold which is open and dense in $\O$.

A \emph{smooth map} between two orbifolds $\O_1$, $\O_2$ is a continuous map $f\co \O_1\to\O_2$ between the underlying topological 
spaces such that for every $x\in \O_1$ there is a chart $(U_1,\widetilde{U}_1/\Gamma_1,\pi_1)$ on $\O_1$ around $x$, a chart 
$(U_2,\widetilde{U}_2/\Gamma_2,\pi_2)$ on $\O_2$ around $f(x)$ and a pair $(\tilde{f},\Theta)$ consisting of a smooth map 
$\tilde{f}\in C^\infty(\widetilde{U}_1,\widetilde{U}_2)$ and a homomorphism $\Theta\co \Gamma_1\to\Gamma_2$ such that 
$\pi_2\circ\tilde{f}=f\circ\pi_1$ and $\tilde{f}(\gamma\tilde{y})=\Theta(\gamma)\tilde{f}(\tilde{y})
\text{ for }\tilde{y}\in\widetilde{U}_1, \gamma\in\Gamma_1$.
A smooth map $f$ where the local lifts $\tilde{f}$ can always be chosen to be submersions is called a \emph{submersion} between 
orbifolds.

Let $\mathfrak{A}$ be a not necessarily maximal atlas on an orbifold $\O$. An \emph{$(r,s)$-tensor field associated with 
$\mathfrak{A}$} is given by a family $\tau=(\tau_\pi)_{\pi\in \mathfrak{A}}$, where for each chart $(U,\widetilde{U}/\Gamma,\pi)$ 
(which we also denote by $\pi$ for short) in $\mathfrak{A}$ the associated element $\tau_\pi$ is a $\Gamma$-invariant $(r,s)$-tensor 
field on $\widetilde{U}$. Moreover, $\tau$ has to satisfy the following compatibility condition: Given charts 
$(U_i,\widetilde{U}_i/\Gamma_i,\pi_i)$, $i=1,2$, in $\mathfrak{A}$ and $x\in U_1\cap U_2$, there is a chart 
$(U,\widetilde{U}/\Gamma,\pi)$ on $\O$ (which need not be in $\mathfrak{A}$) satisfying $x\in U\subset U_1\cap U_2$ together with 
injections $\lambda_1,\lambda_2$ from $\pi$ into $\pi_1$ and $\pi_2$, respectively, such that 
$\lambda_1^*\tau_{\pi_1}=\lambda_2^*\tau_{\pi_2}$ on $\widetilde{U}$.
A tensor field on the maximal atlas of $\O$ is called a \emph{tensor field on $\O$}. It can be shown that a tensor field on an 
arbitrary atlas $\mathfrak{A}$ on $\O$ induces a unique tensor field on $\O$. If $\tau$ is a tensor field on $\O$, we set 
$\tau_\reg:=\tau_\pi$ for $\pi$ given by $(\O^\reg,\O^\reg/\{\id_{\O^\reg}\},\id_{\O^\reg})$. 

A $(1,0)$-tensor field is then called a \emph{vector field} on $\O$ and a $(0,2)$-tensor field consisting of Riemannian metrics is a 
\emph{Riemannian metric} on $\O$. Given a smooth real-valued function $f$ and a Riemannian metric on $\O$, we denote by $\grad f$ 
the vector field on $\O$ given in each chart $\pi$ by the gradient of $f\circ\pi$ with respect to the given metric. Note that given 
vector fields $X^1,\dots,X^k$ and a $(0,k)$-tensor field $\tau$ on $\O$, we can set $f_\pi:=\tau_\pi(X^1_\pi,\dots,X^k_\pi)\in 
C^\infty(\widetilde{U})^\Gamma$ for every chart $(U,\widetilde{U}/\Gamma,\pi)$ on $\O$. Patching the $\ol{f_\pi}\in C(U)$ induced by 
$f_\pi$ together, we obtain a well-defined smooth function $f$ on $\O$ which we will denote by $\tau(X^1,\dots,X^k)$. Given a 
smooth map $f\co \O_1\to\O_2$ between orbifolds and an arbitrary $(0,k)$-tensor field $\tau$ on $\O_2$, the \emph{pull-back} 
$f^*\tau$ as a $(0,k)$-tensor field on $\O_1$ can be defined using the pull-backs
of the components of $\tau$ via local lifts $\tilde{f}$ of $f$. In particular, this applies to a $(0,k)$-tensor field consisting of 
$k$-forms, which we will call \emph{$k$-form on $\O$}.

To integrate on a compact Riemannian orbifold we first introduce \emph{densities} on orbifolds. Let $\O=(X,\mathfrak{A})$ be an 
$n$-dimensional orbifold. In analogy to the case of $n$-forms, a density is given by a family 
$\mu=\{\mu_\pi\}_{\pi\in\mathfrak{A}}$, where for each chart $(U,\widetilde{U}/\Gamma,\pi)$ in $\mathfrak{A}$ the associated element 
$\mu_\pi$ is a $\Gamma_\pi$-invariant density on $\widetilde{U}$. Moreover, we assume that $\mu$ satisfies a compatibility condition 
analogous to the one for tensor fields. Given a density $\mu$ on a compact orbifold $\O$, we can define the integral of $\mu$ over 
$\O$ by
\begin{equation*}
\int_\O \mu := \sum_i \frac{1}{|\Gamma_i|}\int_{\widetilde{U}_i}(\psi_i\circ{\pi_i}) \mu_{\pi_i},
\end{equation*}
where $\{(U_i,\widetilde{U}_i/\Gamma_i,\pi_i)\}_i$ is a finite atlas of $\O$ and $\{\psi_i\}$ is a smooth partition of unity on $\O$ 
such that each $\psi_i\in C^\infty(\O)$ has support in $U_i$. It can be shown that for a diffeomorphism $F\co \O_1\to\O_2$ the 
respective formula for densities on manifolds implies $\int_{\O_1}F^*\mu=\int_{\O_2}\mu$ for every density $\mu$ on $\O_2$.

Given a Riemannian metric $g$ on $\O$, note that for every chart $(U,\widetilde{U}/\Gamma,\pi)$ in $\mathfrak{A}$ the Riemannian 
metric $g_\pi$ defines the Riemannian density $\dvol_{g_\pi}$ on the manifold $\widetilde{U}$. The density 
$\dvol_g:=\{\dvol_{g_\pi}\}_{\pi\in\mathfrak{A}}$ is called the \emph{Riemannian density} on $(\O,g)$. Given a smooth function $f$ 
on $\O$, we can define the integral of $f$ over $\O$ by 
\[\int_\O f:=\int_\O f\dvol_g.\]

We now assume that we are given a compact connected Lie group $G$ acting smoothly and effectively on a manifold $M$ such that the 
stabilizer of every point in $M$ is finite. Denote the canonical projection by $P\co M\to M/G$. Using foliation theory, it can be 
shown that under these conditions the quotient $M/G$ carries a canonical orbifold structure whose restriction to $(M/G)^\reg$ is 
given by the usual manifold structure on the free quotient of points in $M$ with trivial $G$-stabilizers (\cite{MR2012261,MR932463}).

The quotient map $P\co M\to M/G$ becomes a submersion for this or\-bi\-fold structure on $M/G$ and the isotropy of a point in $M/G$ is 
given by the $G$-stabilizer of an arbitrary preimage by $P$. The pull-back $P^*$ gives an isomorphism between $(0,k)$-tensor fields 
on the orbifold $M/G$ and $G$-horizontal $G$-invariant $(0,k)$-tensor fields on $M$. Moreover, given a $G$-invariant Riemannian 
metric $g$ on $M$, we can canonically identify $G$-invariant vector fields on $M$ which are $G$-horizontal with respect to $g$ with 
vector fields on $M/G$. This isomorphism is unique in the sense that it is the unique extension of the usual isomorphism for the 
manifold case given by the differential of the manifold submersion $P_{|M_G}\co  M_G\to (M/G)^\reg$ (where $M_G:=\{x\in 
M;~G_x=\{\id\}\}=P^{-1}((M/G)^\reg)$).

A \emph{Riemannian submersion} is by definition a submersion $f$ between two Riemannian orbifolds such that the local lifts 
$\tilde{f}$ can be chosen to be Riemannian submersions with respect to the given metrics. It can be shown that in the situation of 
the paragraph above, given a $G$-invariant Riemannian metric $g$ on $M$, there is a unique Riemannian metric (the so-called 
\emph{submersion metric}) $g^G$ on $M/G$ such that the canonical projection $P\co (M,g)\to (M/G,g^G)$ becomes a Riemannian 
submersion.

To define \emph{fundamental vector fields} on orbifolds suppose we are given a Lie group $G$ acting smoothly and effectively on an 
orbifold $\O$, denote the action by $\phi\co G\times \O\to\O$ and let $X\in T_eG$ be an element of the Lie algebra of $G$. We define 
a vector field $X^*$ on $\O$ in the following way: Let $x\in \O$. Since $\phi$ is smooth, there are charts 
$(U,\widetilde{U}/\Gamma,\pi)$ and $(U^\prime,\widetilde{U}^\prime/\Gamma^\prime,\pi^\prime)$ of $\O$ over $x$, an open 
neighbourhood $W$ of $e$ in $G$ and a smooth map $\tilde\phi\co W\times \widetilde{U}\to\widetilde{U}^\prime$ such that 
$\pi^\prime\circ\tilde{\phi}=\phi\circ(\id_W,\pi)$. By choosing $U$ sufficiently small around $x$, we can assume that 
$h:=\tilde\phi(e,\cdot)\co \widetilde{U}\to\widetilde{U}^\prime$ is an embedding. Denote the inverse $\widetilde{U}^\prime\supset 
h(\widetilde{U})\to\widetilde{U}$ by $h^{-1}$. Now recall that we had fixed $X\in T_eG$ and define a vector field $\sigma_x(X)$ on 
$\widetilde{U}$ by $\sigma_x(X)({\tilde{y}}):=\frac{d}{dt}_{|t=0}h^{-1}(\tilde\phi(\exp(tX),{\tilde{y}}))$, where $\exp$ denotes the 
Lie group exponential map. It can be shown that the vector fields $\{\sigma_x(X)\}_{x\in\O}$ satisfy the compatibiliy conditions for 
orbifold vector fields and hence induce a unique vector field on $\O$ which we will denote by $X^*$ and call the fundamental vector 
field associated with $X$.

\subsection{The Laplace spectrum}
\label{subsection:orbispec}
Given a Riemannian orbifold $(\O,\langle,\rangle)$, it is possible to generalize the Laplace operator from the manifold case by setting $\Delta f(x):=\widetilde{\Delta}(f\circ\pi)(\tilde{x})$, where $x\in\O$, $(U,\widetilde{U}/\Gamma,\pi)$ is a chart around $x$, $\tilde{x}\in\pi^{-1}(x)$, $f\in C^\infty(\O)$ and $\widetilde{\Delta}$ is the Laplace operator on the Riemannian manifold $(\widetilde{U},\langle,\rangle_\pi)$. The spectrum of the Laplacian on compact orbifolds was first
investigated by Donnelly (\cite{MR537804}). He proved the following theorem for
good orbifolds which was later generalized to arbitrary orbifolds by Chiang (\cite{MR1089240}), also compare \cite{MR2433665}.

\begin{theorem}
\label{theorem:eigenspaces}
Let $(\O,\langle,\rangle)$ be a compact Riemannian orbifold. Then
every eigenvalue of $\Delta$ on $C^\infty(\O)$ has finite
multiplicity and the spectrum $\spec(\O)$ of $\Delta$ consists of a sequence
$0=\mu_0\le\mu_1\le\mu_2\le\cdots$, where
$\mu_i\to\infty$.
Moreover, there is an orthonormal basis $\{\phi_i\}_{i\ge 0}\subset
C^\infty(\O)$ of the Hilbert space
$L^2(\O,\langle,\rangle)$ such that $\Delta\phi_i=\mu_i\phi_i$.
\end{theorem}

Two compact Riemannian orbifolds $\O_1$ and $\O_2$ are called
\emph{isospectral} if $\spec(\O_1)=\spec(\O_2)$ with multiplicities. From now on $\O$ will always denote a compact Riemannian orbifold.

We will need Green's Formula for orbifolds in its complex version. To this end given a 
Riemannian orbifold $(\O,\langle,\rangle)$, let $\Delta^\C$ and $\grad^\C$ denote the complexifications of $\Delta$ and $\grad$, respectively. Moreover, let $\langle,\rangle_\C$ denote the sesquilinear extension of $\langle,\rangle$ to complex-valued vector fields on $\O$. The following lemma then follows directly from the respective formula for smooth functions on manifolds with compact support.

\begin{lemma}
\label{lemma:green}
Let $(\O,\langle,\rangle)$ be a compact Riemannian orbifold and let $f_1,f_2\in C^\infty(\O,\C)$. Then
\[\int_\O f_1\ol{\Delta^\C f_2}=\int_\O \bigl\langle \grad^\C f_1,\grad^\C f_2\bigr\rangle_\C=\int_\O\ol{f_2}\Delta^\C f_1.\]
\end{lemma}

Note that --- as in the manifold setting --- this lemma implies that the eigenvalues of $\Delta^\C$ are real and each eigenspace of $\Delta^\C$ is just given by the complexification of the eigenspace of $\Delta$ associated with the same eigenvalue. In particular, $\{\phi_i\}_{i\ge0}$ from Theorem \ref{theorem:eigenspaces} also gives an orthonormal basis of the space of complex-valued $L^2$-functions, which we will denote by $L^2(\O,\langle,\rangle)$ from now on. Next consider $C^\infty(\O,\C)$ as a Pre-Hilbert-space with the sesquilinear inner product
\[(f_1,f_2)_1=\int_\O f_1\ol{f_2}+\int_\O\langle \grad^\C f_1,\grad^\C f_2\rangle_\C.\]

The (complex) Sobolev space $H^1(\O,\langle,\rangle)$ is the completion of $C^\infty(\O,\C)$ with respect to this inner product. The Rayleigh quotient $R\co H^1(\O,\langle,\rangle)\setminus\{0\}\to [0,\infty)$
is the unique continuous extension of the functional
\[C^\infty(\O,\C)\setminus\{0\}\ni f\mapsto \frac{\int \bigl\langle \grad^\C f, \grad^\C f\bigr\rangle_\C}{\int |f|^2}\in 
[0,\infty)\]
to $H^1(\O,\langle,\rangle)\setminus\{0\}$.
Theorem \ref{theorem:eigenspaces} in its complex form together with Lem\-ma~\ref{lemma:green} imply the following variational characterization. The proof is almost literally the same as in the manifold case (\cite[III.28]{MR861271}), also compare \cite[Lemma 6.3]{MR2155380}.

\begin{theorem}
\label{theorem:variation}
Let $\O$ be a compact Riemannian orbifold and let $L_k$ denote the set of all $k$-dimensional subspaces of $H^1(\O,\langle,\rangle)$. Then
\[\mu_k=\inf_{U\in L_k}\sup_{f\in U\setminus\{0\}}R(f).\]
\end{theorem}

As in the manifold case it can be shown that the spectrum determines the volume, dimension and other geometric properties of an orbifold (\cite{MR2494312,MR1821378}). In order to investigate which properties are not determined by the spectrum, one needs constructions of isospectral (but nonisometric) orbifolds. There are various constructions of isospectral manifolds (see \cite{MR1736857} for an overview), but in the next paragraph we shall briefly summarize only those which have already been generalized to get examples of isospectral singular orbifolds.

Sunada's Theorem (\cite{MR782558}) in its orbifold version by B\'erard (\cite{MR1152950}) was the first construction of isospectral singular orbifolds and was used in \cite{MR1181812} to give examples of isospectral plane domains and in \cite{MR2263484} to construct arbitrarily large families of isospectral orbifolds with pairwise nonisomorphic biggest isotropy groups. Both the Sunada Theorem and an explicit formula for eigenvalues on flat orbifolds (\cite{MR1861302}) can even be used to construct pairs of isospectral orbifolds in which the maximal orders of isotropy groups are different (\cite{MR2447904}). More intricate generalizations of the Sunada Theorem were used in \cite{sutton10} and \cite{MR2761691} to give continuous families of isospectral singular orbifolds. Besides, \cite{MR2802590} generalized results from \cite{MR597742} to construct isospectral orbifold lens spaces.

However, all pairs of isospectral orbifolds above are good. More precisely, they are either of the form $M/\Gamma_1, M/\Gamma_2$ with $\Gamma_i$ discrete subgroups of the isometry group of a Riemannian manifold $M$ or (in the case of \cite{sutton10}) $M_1/\Gamma_1, M_2/\Gamma_2$ with $M_1,M_2$ isospectral compact Riemannian manifolds and each $\Gamma_i$ a finite subgroup of the isometry group of $M_i$. It has been shown in \cite[Proposition 3.4(ii)]{MR2044174} that the first type cannot give an isospectral pair of a Riemannian manifold and a singular orbifold. An analogous argument (also using the heat kernel expansion from \cite{MR2433665}, see \cite{sutton10,weilandt10}) shows that the second type cannot yield such a pair, either.

These observations are the basis for our interest in the spectral geometry of bad orbifolds. The only obvious way to construct isospectral bad orbifolds using known constructions would be to take a pair of good isospectral orbifolds $\O_1,\O_2$ (which can, of course, be manifolds) and a bad orbifold $\O$. Then the Riemannian products $\O_1 \times \O$ and $\O_2\times \O$ are isospectral bad orbifolds. However, in Section \ref{section:examples} we will present the first examples of isospectral bad orbifolds which cannot be written as nontrivial products.

Note that it would be pointless to apply obstructions to the isospectrality of a pair of a manifold and a singular orbifold to the examples in Section \ref{section:examples}, since our isospectral pairs and families are always diffeomorphic by definition. Also note that such obstructions in general do not apply to the case of orbifolds with boundary as it is actually possible to construct a Dirichlet-isospectral pair of a singular orbifold and a manifold (\cite{herbrich11}). For more results on the spectral geometry of (closed) orbifolds see \cite{MR2433665,MR1821378} and the references therein.

\section{The torus method for orbifolds}
\label{section:torus-method}
In this section we generalize the so-called torus method from \cite{MR1895349} to orbifolds.

\subsection{Isospectral metrics}
Let $T$ be a torus (i.e., a nontrivial compact connected abelian Lie group) acting effectively and smoothly on a connected orbifold $\O$. Recall that $\O^\reg$ is connected, open and dense in $\O$. It is also obviously $T$-invariant. Since $T$ is abelian and acts effectively, the (not necessarily connected) manifold $\O_T^\reg:=\{x\in \O^\reg;~G_x=\{\id\}\}$ is open and dense in $\O^\reg$ and hence in $\O$. Given an orbifold metric $g$ on $\O$ we also write $g$ for the induced (manifold) metric on $\O^\reg$ and its submanifolds. Note that in the theorem below we do not assume $\O$ to be oriented and $\dvol_g$ stands for the Riemannian density on $(\O,g)$. In the corresponding proof and later on we also need the following notation: Let $\mathfrak{t}:=T_eT$ denote the Lie algebra of $T$. Setting $\mathcal{L}:=\ker(\exp\co \mathfrak{t}\to T)$, we observe that $\exp$ induces an isomorphism from $\mathfrak{t}/\mathcal{L}$ to $T$. We will write $\mathcal{L}^*:=\{\phi\in\mathfrak{t}^
 *;~\phi(x)\in\Z~\forall X\in\mathcal{L}\}$ for the dual lattice.

\begin{theorem}
\label{theorem:torusaction}
Let $T$ be a torus acting effectively and isometrically on two compact connected Riemannian orbifolds $(\O,g)$ and $(\O^\prime, g^\prime)$. Set $\widehat{\O}=\O_T^\reg$, $\widehat{\O}^\prime={\O^\prime}_T^\reg$. Assume that for every subtorus $W\subset T$ of codimension $1$ there is a $T$-equivariant diffeomorphism $F_W\co \O\to \O^\prime$ satisfying $F_W^*\dvol_{g^\prime}=\dvol_g$ which induces an isometry between the manifolds $(\widehat{\O}/W,g^W)$ and $(\widehat{\O}^\prime/W,{g^\prime}^W)$. Then the orbifolds $(\O,g)$ and $(\O^\prime,g^\prime)$ are isospectral.
\end{theorem}

\begin{proof}
  Consider the Sobolev spaces $H:=H^1(\O,g)$ and $H^\prime:=H^1(\O^\prime,g^\prime)$. One can construct an isometry $H^\prime\to H$ preserving $L^2$-norms in the same way as has been done in the proof of \cite[Theorem 1.4]{MR1895349} for the manifold setting: 

Consider the unitary representation of $T$ on $H$ given by $(zf)(x):= f(zx)$ for $z\in T$, $f\in H$, $x\in \O$. Then the  $T$-module $H$ can be written as the Hilbert sum $H=\bigoplus_{\mu\in{\mathcal L}^*}H_\mu$
of $T$-modules \[H_\mu=\bigl\{f\in H;~[Z] f=e^{2\pi i \mu(Z)}f~\forall Z\in {\mathfrak t}\bigr\}.\] 
For each subtorus $W$ of $T$ of codimension 1 set
\[E_W:=\bigoplus_{\substack{\mu\in{\mathcal L}^*\setminus\{0\}\\ T_eW=\ker\mu}}H_\mu\]
and denote the (Hilbert) sum over all these subtori by $\bigoplus_W$. We obtain the decomposition
\[H=H_0\oplus \bigoplus_{\mu\in{\mathcal L}^*\setminus\{0\}}H_\mu=H_0\oplus\bigoplus_W E_W.\]
Moreover, set 
\[H_W:=H_0\oplus E_W =\bigoplus_{\substack{\mu\in{\mathcal L}^*\\ T_eW\subset \ker\mu}}H_\mu\]
and note that $H_W$ consists precisely of the $W$-invariant functions in $H$.

Now use the analogous notation $H^\prime_\mu, E^\prime_W, H^\prime_W$ for the corresponding subspaces of $H^\prime$. Fix a subtorus 
$W$ of $T$ of codimension 1 and let $F_W\co \O\to \O^\prime$ be the corresponding diffeomorphism from the assumption. We will show 
that $F_W^*\co H_W^\prime\to H_W$ is a Hilbert space isometry preserving the $L^2$-norm. It preserves the $L^2$-norm since 
$F_W^*\dvol_{g^\prime}=\dvol_g$. Now let $\psi\in C^\infty(\O^\prime)$ be invariant under $W$. Since the map 
$(\widehat{\O}/W,g^W)\to(\widehat{\O}^\prime/W,{g^\prime}^W)$ induced by $F_W$ is an isometry and the quotient maps are Riemannian 
submersions, we have $\|\grad(\psi\circ F_W)_{F_W^{-1}(y)}\|_g=\|\grad\psi_y\|_{g^\prime}$ for all $y\in\widehat{\O}$. Since 
$\widehat{\O}$ is dense in $\O$ and $\widehat{\O}^\prime$ is dense in $\O^\prime$, this implies that $F_W^*\co H_W^\prime\to H_W$ is 
a Hilbert space isometry with respect to the $H^1$-product. Hence so is its restriction ${F_W^*}_{|E_W^\prime}\co  E_W^\prime\to 
E_W$.

Considering all subtori $W\subset T$ of codimension 1 and choosing an arbitrary $F_W^*\co H_0^\prime\to H_0$, we obtain an 
$L^2$-norm-preserving isometry $H^\prime\to H$. Isospectrality of $(\O,g)$ and $(\O^\prime,g^\prime)$ finally follows from Theorem 
\ref{theorem:variation}.
\end{proof}

We will need the following definitions and results, which generalize \cite[1.5]{MR1895349} to our orbifold setting.

\begin{notations}
  \label{notations:O-T}
  We now fix a torus $T$ and use the notation ${\mathfrak t}=T_eT$, ${\mathcal L}:=\ker(\exp\co {\mathfrak t}\to T)$ as above. 
  Moreover, fix a compact connected Riemannian orbifold $(\O,g_0)$ and a smooth effective action of $T$ on $(\O,g_0)$ by isometries 
  and set $\widehat{\O}:=\O^\reg_T$. If $Z\in{\mathfrak t}$, write $Z^{\widehat{~}}:=\widehat{Z}:={Z^*_\reg}_{|\widehat{\O}}$ for 
  the fundamental vector field on $\widehat{\O}$ induced by $Z$.
\begin{enumerate}
 \item \label{enumremark:admissible} A ${\mathfrak t}$-valued $1$-form on $\O$ will be called \emph{admissible} if it is $T$-horizontal (i.e., it vanishes on all $X^*$, $X\in\mathfrak{t}$) and $T$-invariant.
 \item \label{enumremark:glambda} For an admissible $1$-form $\lambda$ on the orbifold $\O$ denote by $g_\lambda$ the Riemannian metric given in each chart $(U,\widetilde{U}/\Gamma,\pi)$ on $\O$ by
\[{g_\lambda}_\pi(X,Y)={g_0}_\pi(X+(\lambda_\pi(X))_\pi^*,Y+(\lambda_\pi(Y))_\pi^*)\]
for vector fields $X,Y\in{\mathcal V}(\widetilde{U})$. It is not hard to verify that this indeed defines a Riemannian orbifold metric.

Note that if $\Phi_{\lambda,\pi}$ denotes the $C^\infty(\widetilde{U})$-isomorphism
\[{\mathcal V}(\widetilde{U})\ni X\mapsto X-(\lambda_\pi(X))_\pi^*\in \mathcal{V}(\widetilde{U}),\] then 
${g_\lambda}_\pi=(\Phi_{\lambda,\pi}^{-1})^*{g_0}_\pi$. Since $\lambda$ is horizontal, $\Phi_{\lambda,\pi}$ is unipotent and this 
implies $\dvol_{g_{\lambda, \pi}}=\bigl|\det\Phi_{\lambda,\pi}^{-1}\bigr|\dvol_{g_{0,\pi}}=\dvol_{g_{0,\pi}}$. Since this holds for every 
chart $\pi$, we have $\dvol_{g_\lambda}=\dvol_{g_0}$.

\item $g_\lambda$ is $T$-invariant: let $z\in T$, $x\in\O$. There are charts $(U_i,\widetilde{U}_i/\Gamma_i,\pi_i)$, $i=1,2$, on $\O$ 
and a diffeomorphism $\widetilde{z}\in C^\infty(\widetilde{U}_1,\widetilde{U}_2)$ such that $x\in U_1$, $zx\in U_2$ and 
$\pi_2\circ\widetilde{z}=z\circ\pi_1$. Since $T$ is abelian, fundamental vector fields on $\O$ are $T$-invariant; in particular, 
$Z_{\pi_1}^*\in{\mathcal V}(\widetilde{U}_1)$ is $\widetilde{z}$-related to $Z_{\pi_2}^*\in{\mathcal V}(\widetilde{U}_2)$ for every 
$Z\in\mathfrak{t}$. Using that $\lambda$ and $g_0$ are also $T$-invariant, a straightforward calculation 
shows 
$\widetilde{z}^*g_{\lambda,\pi_2}=g_{\lambda,\pi_1}$.

\item Moreover, note that for every $x\in \widehat{\O}$ the metric $g_\lambda$ on $T_x\widehat{\O}$ restricts to the same metric as $g_0$ on the vertical subspace ${\mathfrak t}_x=\{\widehat{Z}_x;Z\in{\mathfrak t}\}\subset T_x\widehat{\O}$, because $\lambda$ is $T$-horizontal. 
Moreover, note that the metrics $g_0^T$ and $g_\lambda^T$ on $\widehat{\O}/T$ coincide. 
\end{enumerate}
\end{notations}

The proof of the next theorem is now just an imitation of the proof of \cite[Theorem 1.6]{MR1895349}.

\begin{theorem}
\label{theorem:isospect}
Let $\lambda$, $\lambda^\prime$ be two admissible $1$-forms on $\O$ satisfying: 

For every $\mu\in{\mathcal L}^*$ there is a $T$-equivariant isometry $F_\mu$ on $(\O,g_0)$ such that 
\begin{equation}
\label{equation:Fmu}
\mu\circ\lambda=F_\mu^*(\mu\circ\lambda^\prime).
\end{equation}

Then $(\O,g_\lambda)$ and $(\O,g_{\lambda^\prime})$ are isospectral.
\end{theorem}

\begin{proof}
We shall use Theorem \ref{theorem:torusaction}. So let $W$ be a subtorus of $T$ of codimension 1 and choose $\mu\in{\mathcal L}^*$ such that $\ker\mu=T_eW$. Let $F_\mu\in\Isom(\O,g_0)$ be a corresponding $T$-invariant isometry satisfying \eqref{equation:Fmu}. We will show that $F_W:=F_\mu$ satisfies the conditions of Theorem \ref{theorem:torusaction}: Since $F_\mu$ is an isometry, we have by the remarks above that $F_\mu^*\dvol_{g_{\lambda^\prime}}=F_\mu^*\dvol_{g_0}=\dvol_{g_0}=\dvol_{g_\lambda}$. To see that $F_\mu$ induces an isometry between the manifolds $(\widehat{\O}/W,g^W_\lambda)$ and $(\widehat{\O}/W,{g}^W_{\lambda^\prime})$, let $x\in\widehat{\O}$, let $V\in T_x\widehat{\O}$ be $W$-horizontal with respect to $g_\lambda$ and set $X:=\widehat{\Phi}_\lambda^{-1}(V)\in T_x\widehat{\O}$, $Y:=\widehat{\Phi}_{\lambda^\prime}({F_\mu}_*X)\in T_{F_\mu(x)}\widehat{\O}$.

First, note that ${F_\mu}_*V-Y$ is $W$-vertical: Condition \eqref{equation:Fmu} implies that $Z:=\widehat{\lambda^\prime}({F_\mu}_*X)-\widehat{\lambda}(X)\in \ker\mu$. Using that $F_\mu$ is $T$-equivariant, it is straightforward to show that for $Y:=\widehat{\Phi}_{\lambda^\prime}({F_\mu}_*X)$ we obtain ${F_\mu}_*V-Y=\widehat{Z}_{F_\mu(x)}$, which is $W$-vertical by our choice of $\mu$.

Second, $Y$ is $W$-horizontal with respect to $g_{\lambda^\prime}$: Since $\lambda$ is $T$-horizontal and $V$ is $W$-horizontal with respect to $g_\lambda$, the vector $X=\widehat{\Phi}_\lambda^{-1}(V)\in T_x\widehat{\O}$ is $W$-horizontal with respect to $g_0$. Hence so is ${F_\mu}_*X$, and $Y=\widehat{\Phi}_{\lambda^\prime}({F_\mu}_*X)$ is $W$-horizontal with respect to $g_{\lambda^\prime}$.

The two observations above imply that $Y$ is the $W$-horizontal part of ${F_\mu}_*V$ with respect to $g_{\lambda^\prime}$. Since
$\|Y\|_{g_{\lambda^\prime}}=\|{F_\mu}_*X\|_{g_0}=\|X\|_{g_0}=\|V\|_{g_\lambda}$, we conclude that $F_\mu$ indeed induces an isometry between $(\widehat{\O}/W,g^W_\lambda)$ and $(\widehat{\O}/W,{g}^W_{\lambda^\prime})$.
\end{proof}

\subsection{Nonisometry}
\label{subsection:general-nonisometry}
In this section we give a sufficient criterion that implies two orbifolds $(\O,g_\lambda)$ and $(\O,g_{\lambda^\prime})$ as in Theorem \ref{theorem:isospect} are not isometric. Let $(\O,g_0)$, $T$, ${\mathfrak t}$, ${\mathcal L}$, $\widehat{\O}$ be as in the preceding section. Note that the action of $T$ on $\widehat{\O}$ gives $\widehat{\O}$ the structure of a principal $T$-bundle $\pi\co \widehat{\O}\to\widehat{\O}/T$. By $\lambda$ we denote an admissible ${\mathfrak t}$-valued $1$-form on $\O$. We now recall the notations and remarks from \cite[2.1]{MR1895349}, applied to our special case of the connected $T$-invariant manifold $\O^\reg$.

\begin{notations}
\label{proposition:nonisolist}\leavevmode
\begin{enumerate}
\item \label{proposition:nonisolisti} A diffeomorphism $F\co \O^\reg\to\O^\reg$ is called \emph{$T$-preserving} if conjugation by $F$ preserves $T\subset\Diffeo(\O^\reg)$, i.e., \[c^F(z):=F\circ z\circ F^{-1}\in T~\forall z\in T.\] In this case we denote by $\Psi_F:=c^F_*$ the automorphism of ${\mathfrak t}=T_eT$ induced by the isomorphism $c^F$ on $T$. Obviously, each $T$-preserving diffeomorphism $F$ of $\O^\reg$ maps $T$-orbits to $T$-orbits. In particular, $F$ preserves $\widehat{\O}$. Moreover, it is straightforward to show $F_*\widehat{Z}=\widehat{\Psi_F(Z)}$ for all $Z\in\mathfrak{t}$.

\item \label{proposition:nonisolist-Aut} We denote by $\Aut^T_{g_0}(\O^\reg)$ the group of all $T$-preserving diffeomorphisms of $\O^\reg$ which, in addition, preserve the $g_0$-norm of vectors tangent to the $T$-orbits in $\widehat{\O}$ and induce an isometry of the Riemannian manifold $(\widehat{\O}/T,g_0^T)$. We denote the corresponding group of induced isometries by $\ol{\Aut}^T_{g_0}(\O^\reg)\subset \Isom(\widehat{\O}/T,g^T_0)$.

\item \label{proposition:nonisolist-D} We define 
\[{\mathcal D} := \{\Psi_F;~F\in\Aut^T_{g_0}(\O^\reg)\}\subset\Aut({\mathfrak t}).\] 
Note that ${\mathcal D}$ is discrete because it is a subgroup of the discrete group 
$\{\Psi\in \Aut({\mathfrak t});~\Psi({\mathcal L})={\mathcal L}\}$.

\item \label{propositon:nonisolist-iv} Let $\omega_0\co T\widehat{\O}\to {\mathfrak t}$ denote the connection form on the principal $T$-bundle $\widehat{\O}$ associated with $g_0$; i.e., $\omega_0(\widehat{Z})=Z~\forall Z\in{\mathfrak t}$ and for each $x\in\widehat{\O}$ the kernel $\ker({\omega_0}_{|T_x\widehat{\O}})$ is the $g_0$-orthogonal complement of the vertical space ${\mathfrak t}_x=\{\widehat{Z}_x;~Z\in{\mathfrak t}\}$ in $T_x\widehat{\O}$. The connection form on $\widehat{\O}$ associated with $g_\lambda$ is easily seen to be given by $\omega_\lambda:=\omega_0+\hat{\lambda}$.

\item Let $\Omega_\lambda$ denote the curvature form on the manifold $\widehat{\O}/T$ associated with the connection form $\omega_\lambda$ on $\widehat{\O}$. We have $\pi^*\Omega_\lambda=d\omega_\lambda$, because $T$ is abelian.

\item \label{proposition:nonisolist-vi} Since $\widehat{\lambda}$ is $T$-invariant and $T$-horizontal, it induces a ${\mathfrak t}$-valued $1$-form $\ol{\lambda}$ on $\widehat{\O}/T$. Then $\pi^*\Omega_\lambda=d\omega_\lambda=d\omega_0+d\widehat{\lambda}$ implies $\Omega_\lambda=\Omega_0+d\ol{\lambda}$.
\end{enumerate}
\end{notations}

\begin{lemma}
\label{lemma:fomega}
Let $F\co (\O^\reg,g_\lambda)\to (\O^\reg,g_{\lambda^\prime})$ be a $T$-preserving isometry. Then:
\begin{enumerate}
\item\label{lemma:fomega.i} 
$F$ preserves the $g_0$-norm of vectors tangent to the $T$-orbits in $\widehat{\O}$, and it induces an isometry $\bar{F}$ of 
$(\widehat{\O}/T,g_0^T)$. In particular, $F\in\Aut_{g_0}^T(\O)$ and $\Psi_F\in{\mathcal D}$.
\item\label{lemma:fomega.ii}
 $F^*\omega_{\lambda^\prime}=\Psi_F\circ \omega_\lambda\in\Omega^1(\widehat{\O},{\mathfrak t})$, in particular $F^* 
 d\omega_{\lambda^\prime}=\Psi_F\circ d\omega_\lambda$.
\item\label{lemma:fomega.iii} 
The isometry $\bar{F}$ of $(\widehat{\O}/T,g_0^T)$ satisfies $\bar{F}^*\Omega_{\lambda^\prime}=\Psi_F\circ \Omega_\lambda$.
\end{enumerate}
\end{lemma}
\begin{proof}
  Apply \cite[Lemma 2.2]{MR1895349} to the manifold $M:=\O^\reg$.
\end{proof}

Before coming to the following propositions, note that the isometry group $\Isom(\O,g)$ of a Riemannian orbifold $(\O,g)$ endowed with the compact-open topology admits a unique smooth structure that turns it into a Lie group (\cite{MR2355369}). The proof of the following proposition is analogous to the proof of \cite[Proposition 2.3]{MR1895349}.

\begin{proposition}
\label{proposition:G-maxtorus}
Let $\lambda$ be an admissible ${\mathfrak t}$-valued $1$-form on $\O$ such that the associated curvature form $\Omega_\lambda$ on $\widehat{\O}/T$ satisfies the following genericity condition:
\begin{equation}\label{eqG}
\text{No nontrivial $1$-parameter group in $\ol{Aut}_{g_0}^T(\O^\reg)$ preserves $\Omega_{\lambda}$.}\tag{G}
\end{equation}
Then $T$ is a maximal torus in $\Isom(\O,g_\lambda)$
\end{proposition}

\addtocounter{equation}{1}

\begin{proof}
Assume that $F_t\in\Isom(\O,g_\lambda)$ is a 1-parameter family of isometries commuting with $T$. If we can show that $F_t\in 
T~\forall t$, we know that $T$ is maximal. Since the $F_t$ commute with $T$, they are $T$-preserving. By 
Lemma \ref{lemma:fomega}\eqref{lemma:fomega.i} 
the restrictions ${F_t}_{|\O^\reg}$ induce a $1$-parameter family $\bar{F}_t\in\Isom(\widehat{\O}/T,g_0^T)$, hence $F_t\in 
\Aut^T_{g_0}(\O^\reg)$ and $\Psi_{F_t}\in{\mathcal D}~\forall t$. Since $\Psi_{F_0}=\Psi_{\Id}=\Id$ and ${\mathcal D}$ is discrete, 
we have $\Psi_{F_t}=\Id$ for all $t$ and hence by 
Lemma \ref{lemma:fomega}\eqref{lemma:fomega.iii} each $\bar{F}_t$ preserves $\Omega_\lambda$. By \eqref{eqG} 
this implies $\bar{F}_t=\id$ for all $t$. Hence each ${F_t}_{|\widehat{\O}}$ is a gauge transformation of the principal bundle 
$\widehat{\O}\to \widehat{\O}/T$ which, by Lemma \ref{lemma:fomega}\eqref{lemma:fomega.ii}, preserves the connection form $\omega_\lambda$. Therefore 
${F_t}_{|\widehat{\O}}$ acts as an element of $T$ on every connected component of $\widehat{\O}$ (\cite[Lemma 4.2.3]{weilandt10}). 
Since the isometry ${F_t}_{|\O^\reg}$ is determined uniquely by its values on an open set in the connected manifold $\O^\reg$ and 
$\O^\reg$ is dense in $\O$, we conclude that $F_t\in T$.
\end{proof}

Lemma \ref{lemma:fomega} and the proposition above now imply the following proposition. Its proof is almost literally the same as that of \cite[Proposition 2.4]{MR1895349}  but we include it for completeness. Note that we use the fact that the isometry group of a compact orbifold is compact (\cite{MR0355886}).

\begin{proposition}
\label{proposition:nonisometry}
Let $\lambda,\lambda^\prime$ be admissible $1$-forms on $\O$ such that $\Omega_{\lambda^\prime}$ has property \eqref{eqG}. Furthermore, assume that
\begin{equation}\label{eq:N}
\Omega_\lambda\notin {\mathcal D}\circ {\ol{\Aut}_{g_0}^T}(\O^\reg)^*\Omega_{\lambda^\prime}.\tag{N}
\end{equation}
Then $(\O,g_\lambda)$ and $(\O,g_{\lambda^\prime})$ are not isometric.
\end{proposition}

\begin{proof}
Suppose that there were an isometry $F\co (\O,g_\lambda)\to(\O,g_{\lambda^\prime})$. By Proposition \ref{proposition:G-maxtorus}, $T$ is a maximal torus in $\Isom(\O,g_{\lambda^\prime})$. Since $\{F\circ z\circ F^{-1};~z\in T\}$ also is a torus in $\Isom(\O,g_{\lambda^\prime})$ and all maximal tori are conjugate, we can assume $F$ --- after possibly combining it with an isometry of $(\O,g_{\lambda^\prime})$ --- to be $T$-preserving. But then Lemma \ref{lemma:fomega} implies $\bar{F}^*\Omega_{\lambda^\prime}=\Psi_F\circ\Omega_\lambda$ with $\bar{F}\in{\ol{\Aut}_{g_0}^T}(\O^\reg)$ and $\Psi_F\in{\mathcal D}$, which contradicts our assumption.
\end{proof}

\section{Examples of isospectral bad orbifolds}
\label{section:examples}
As mentioned in Section \ref{subsection:orbispec}, one can easily obtain examples of isospectral bad orbifolds of the form $\O\times\O_1,\O\times\O_2$ from isospectral good orbifolds $\O_1,\O_2$ and a bad orbifold $\O$. However, in this section we will use the constructions from the preceding section to give genuinely new examples of isospectral bad orbifolds. More precisely, for every fixed $n\ge 4$ and coprime positive integers $p,q$ we will give isospectral pairs and even families of metrics on certain $2n$-dimensional weighted projective spaces (depending on $p,q$). The latter turn out to be bad orbifolds for $(p,q)\ne (1,1)$.

\subsection{Our weighted projective spaces}
\label{subsection:wps}
Consider the following orbifold which is a special weighted projective space: for $n\ge 4$, let $S^{2n+1}\subset \C^{n+1}$ denote the standard sphere and let $p,q$ be coprime positive integers. Let $S^1\subset \C$ act smoothly on $S^{2n+1}$ by
\begin{equation}
\label{equation:defwp}
\sigma(u,v)=(\sigma^pu,\sigma^qv),
\end{equation}
where $\sigma\in S^1\subset\C$, $u\in \C^{n-1}$, $v\in\C^2$. The quotient $\O:=\O(p,q):=S^{2n+1}/S^1$ under this action becomes an orbifold and \[\O^\reg=\{[(u,v)]\in \O;~u\ne 0\wedge v\ne 0\}\rm{:}\] the points of the form $(u,0)$ are fixed precisely by the $p$-th roots of unity, the points of the form $(0,v)\in S^{2n+1}$ are fixed precisely by the $q$-th roots of unity, and the action is free in all other points. 

For every pair $(p,q)$ we will construct isospectral metrics on the orbifold $\O=\O(p,q)$. Note that for $p=q=1$ we have $\O=\CP^n$. All other orbifolds in this family are singular.

Since $S^{2n+1}$ is simply connected and $S^1$ is connected, the orbifold fundamental group $\pi_1^\text{Orb}(\O(p,q))$ is trivial for all $p,q$ (\cite[Proposition 1.54]{MR2359514}) and hence the orbifolds $\O(p,q)$ for $(p,q)\ne (1,1)$ are ``bad'', i.e., they cannot be written as a quotient of a manifold by a properly discontinuous group action.

Throughout this section, $\langle,\rangle$ will always denote the canonical metric on $S^{2n+1}\subset \C^{n+1}$ given by the restriction of the inner product
\[\langle X,Y\rangle =\re \left(\sum_{i=1}^{n+1} X_i\bar{Y}_i\right)\text{ for }X,Y\in\C^{n+1}.\]
Besides, $\langle,\rangle$ will also denote the unique metric on $\O=S^{2n+1}/S^1$ with respect to which the quotient map $P\co S^{2n+1}\to S^{2n+1}/S^1$ becomes a Riemannian orbifold submersion. In cases where the metric is not specified, we will always assume that $\langle,\rangle$ is used. The metric $\langle,\rangle$ on $\O$ will also be denoted by $g_0$. 

Note that isospectral families of metrics on $\O(1,1)=\CP^n$ have already been given in \cite{rueckriemen06} using the manifold version of the construction in the following section. Similar methods have also led to examples of isospectral families of good orbifolds (\cite{sutton10}). For results on the spectral geometry of weighted projective spaces with their standard metric see \cite{MR2395193} and \cite{MR2529473}.

\begin{remark}
  The results from Section \ref{subsection:isometrics} easily generalize to the case that $p,q_1,q_2$ are natural numbers with greatest common divisor one and $S^1$ acts on $S^{2n+1}$ via $\sigma(u,v_1,v_2):=(\sigma^pu,\sigma^{q_1}v_1,\sigma^{q_2}v_2)$. However, the nonisometry proof (in particular, the statements from Lemma \ref{lemma:Oab} onwards) would become considerably more complicated, and so we restricted our attention to the special case $q_1=q_2$ given in \eqref{equation:defwp}.
\end{remark}

\subsection{Isospectral metrics}
\label{subsection:isometrics}
In this section we will give isospectral metrics on the orbifold $\O=\O(p,q)$. To this end we will apply the torus method from Section \ref{section:torus-method} to a certain action of some quotient of $S^1\times S^1\subset \C^2$ on $\O$. We identify $\R^2$ with ${\mathfrak t}=T_{(1,1)}(S^1\times S^1)$ via $\R^2\ni(t_1,t_2)\mapsto (it_1,it_2)\in {\mathfrak t}\subset\C^2$ and set $Z_1=(i,0),Z_2=(0,i)\in {\mathfrak t}$.

In order to introduce appropriate admissible $1$-forms $\lambda,\lambda^\prime$ on $\O$ we will need the following variation of \cite[Definition 3.2.4]{MR1895349}. (The only difference is a broader definition of equivalence in (\ref{definition:equivalent}).) 

\begin{definition}
\label{definition:isomaps}
Let $j,j^\prime\co  {\mathfrak t}\simeq\R^2\to \mathfrak{su}(m)$ be two linear maps.
\begin{enumerate}
\item
  \label{definition:isospectral} We call $j$ and $j^\prime$ \emph{isospectral} if for each $Z\in {\mathfrak t}$ there is $A_Z\in SU(m)$ such that $j_Z^\prime=A_Zj_ZA_Z^{-1}$.
 \item \label{definition:equivalent} Let $Q\co \C^m\to\C^m$ denote complex conjugation and set
\[{\mathcal E}:=\{\phi\in\Aut({\mathfrak t});~\phi(Z_k)\in\{\pm Z_1,\pm Z_2\}\text{ for }k=1,2\}.\]
We call $j$ and $j^\prime$ \emph{equivalent} if there is $A\in SU(m)\cup SU(m)\circ Q$ and $\Psi\in{\mathcal E}$ such that $j_Z^\prime=Aj_{\Psi(Z)}A^{-1}$ for all $Z\in {\mathfrak t}$.
\item
  \label{definition:generic}We say that $j$ is \emph{generic} if no nonzero element of $\mathfrak{su}(m)$ commutes with both $j_{Z_1}$ and $j_{Z_2}$.
\end{enumerate}
\end{definition}

Note that all properties above are stable under multiplication of both $j$ and $j^\prime$ with a fixed nonzero real number. We will need the following proposition which is just a simplified form of \cite[Proposition 3.2.6(i)]{MR1895349}.
\begin{proposition}
\label{proposition:isomaps}
For every $m\ge 3$ there is an open interval $I\subset \R$ and a continuous family $j(t)$, $t\in I$, of linear maps $\R^2\to \mathfrak{su}(m)$ such that:
\begin{enumerate}
\item The maps $j(t)$ are pairwise isospectral.
\item \label{enumitem:nonequiv} For $t_1,t_2\in I$ with $t_1\ne t_2$ the maps $j(t_1)$ and $j(t_2)$ are not equivalent.
\item All maps $j(t)$ are generic.
\end{enumerate}
\end{proposition}

\begin{remark}
Note that the proof of \eqref{enumitem:nonequiv} in \cite{MR1895349} still holds for our slightly different definition of equivalence, since Definition \ref{definition:isomaps}\eqref{definition:equivalent} still implies that $\tr((j_{Z_1}^2+j_{Z_2}^2)^2)=\tr(({j^\prime_{Z_1}}^2+{j^\prime_{Z_2}}^2)^2)$.
\end{remark}

\subsubsection{Isospectral Pairs}
\label{subsubsection:isospectral-pairs}
In this section we will explain how two isospectral maps $j,j^\prime\co  \R^2\to \mathfrak{su}(n-1)$ (which do not necessarily have to lie in a continuous family) induce isospectral metrics on our orbifold $\O=\O(p,q)$ from Section \ref{subsection:wps}. More precisely, we will describe a construction process which associates metrics $g_{\lambda}, g_{\lambda^\prime}$ on $\O$ with $j,j^\prime$.

Consider the following action of the two-torus $\widetilde{T}:=S^1\times S^1\subset \C^2$ on $S^{2n+1}\subset \C^{n+1}$:
\begin{equation}
\label{equation:T-action}
(\sigma_1,\sigma_2)(u,v_1,v_2)=(u,\sigma_1v_1,\sigma_2v_2)
\end{equation}
for $\sigma_1,\sigma_2\in S^1\subset \C,u\in \C^{n-1}$ and $v_1,v_2\in\C$. This action is isometric and commutes with the $S^1$-action above and hence induces a smooth $\widetilde{T}$-action on $\O$. This action is not effective but induces an effective action of 
\[T:=(S^1\times S^1)/\{(\sigma,\sigma);\sigma\text{ $p$-th root of untity}\}.\]

Note that the exponential map $\mathfrak{t}\ni s_1Z_1+s_2Z_2\mapsto (e^{is_1},e^{is_2})\in \widetilde{T}$ induces an isomorphism $\mathfrak{t}/\widetilde{\mathcal L}\simeq\widetilde{T}$ with $\widetilde{\mathcal L}:=\spann_\Z\{2\pi Z_1,2\pi Z_2\}$ and an isomorphism $\mathfrak{t}/{\mathcal L}\simeq T$ with ${\mathcal L}:=\spann_\Z\{2\pi Z_1,\frac{2\pi}{p}(Z_1+Z_2)\}$.

Moreover, set
\[\widehat{S^{2n+1}}=\{(u,v)\in\C^{n-1}\times\C^2;~\|u\|^2+\|v\|^2=1, u\ne 0, v_j\ne 0~\forall j=1,2\}.\]
With $\widehat{\O}$ defined as in Notations and Remarks \ref{notations:O-T} (with respect to our effective $T$-action on $\O=\O(p,q)$) we then have $\widehat{\O}=P(\widehat{S^{2n+1}})$.

Given a linear map $j\co \R^2\to \mathfrak{su}(n-1)$, define an $\R^2$-valued $1$-form $\kappa=(\kappa^1,\kappa^2)$ on $S^{2n+1}\subset\C^{n+1}$ by
\begin{equation}
\kappa_{(u,v)}^k(U,V):=\|u\|^2\langle j_{Z_k}u,U\rangle -\langle U,iu\rangle \langle j_{Z_k}u,iu\rangle \label{equation:kappa}
\end{equation}
for $u\in\C^{n-1}$, $v\in\C^2$, $U\in\C^{n-1}$ and $V\in\C^2$ and restricting to $S^{2n+1}$. Since $\kappa$ is independent of $V$, it is $T$-horizontal; in particular, $\kappa_{(u,v)}(0,iv)=\kappa_{(u,v)}(Z_1^*+Z_2^*)=0$ for $(u,v)\in S^{2n+1}$. 
Moreover, 
\[\kappa^k_{(u,v)}(iu,0)=\|u\|^2\langle j_{Z_k}u,iu\rangle -\langle iu,iu\rangle\langle j_{Z_k}u,iu\rangle =0\text{ for }k=1,2\]
(as already noted in the proof of \cite[3.2.2]{MR1895349}). Hence $\kappa$ is also $S^1$-horizontal, since the vertical space in $(u,v)\in S^{2n+1}$ under the $S^1$-action is given by the real span of $(ipu,iqv)$. Moreover, $\kappa$ is $S^1$-invariant, since $S^1$ acts isometrically and each $j_{Z_k}\in \mathfrak{su}(n-1)$ commutes with scalars in $S^1\subset \C$.

Note that this implies that $\kappa$ induces a unique $\R^2$-valued $1$-form $\lambda$ on $\O$ satisfying
\begin{equation}
\label{equation:lambdakappa}
P^*\lambda=\kappa.
\end{equation}
Moreover, since $P^*$ commutes with $d$, we have
\[d\lambda(P_*(U_1,V_1),P_*(U_2,V_2))=d\kappa((U_1,V_1),(U_2,V_2)).\]

We will need the following basic observations.

\begin{proposition}
\label{proposition:kappalambda}\leavevmode
\begin{enumerate}
\item \label{proposition:kappalambda:1} $P_{|\widehat{S^{2n+1}}}\co \widehat{S^{2n+1}}\to\widehat{\O}$ is $\widetilde{T}=S^1\times 
S^1$-equivariant.
\item For every $Z\in\mathfrak{t}$ the differential $P_*$ maps the fundamental vector field $Z^*_{|\widehat{S^{2n+1}}}\in{\mathcal 
V}(\widehat{S^{2n+1}})$ to the fundamental vector field $\widehat{Z}$ on $\widehat{\O}$.
\item Let $j\co \mathfrak{t}\simeq \R^2\to\mathfrak{su}(n-1)$ be a linear map. Then for the $\mathfrak{t}$-valued $1$-forms 
$\kappa$ given in \eqref{equation:kappa} and $\lambda$ given in \eqref{equation:lambdakappa} we have:
\begin{enumerate}
\item $\kappa$ is $\widetilde{T}$-invariant and $\widetilde{T}$-horizontal.
\item $\lambda$ is admissible in the sense of\/ \rmref{notations:O-T}\eqref{enumremark:admissible} with respect to the effective 
$T$-action on $\O$ induced by \eqref{equation:T-action}.
\end{enumerate}
\end{enumerate}
\end{proposition}

The following theorem, which partly generalizes \cite[Proposition 3.2.5]{MR1895349}, is now the main result of this section. 
Together with the results in Section \ref{subsection:nonisometry} it implies the existence of nontrivial pairs and families of 
isospectral metrics on $\O=\O(p,q)$.

\begin{theorem}
\label{theorem:isoorbis}
Let $j, j^\prime\co \R^2\to \mathfrak{su}(n-1)$ be isospectral linear maps and let $\lambda$ and $\lambda^\prime$ be the 
corresponding admissible $1$-forms on $\O$ given above. Then $(\O,g_\lambda)$ and $(\O,g_{\lambda^\prime})$ are isospectral 
orbifolds.
\end{theorem}

\begin{proof}
To apply Theorem \ref{theorem:isospect} let $\mu\in {\mathcal L}^*\subset{\mathfrak t}^*$ and set 
\[Z:=\mu(Z_1)Z_1+\mu(Z_2)Z_2\in \mathfrak t.\] 
Then since $j$ and $j^\prime$ are isospectral, we can choose $A_Z\in SU(n-1)$ as in Definition 
\ref{definition:isomaps}(\ref{definition:isospectral}) and set $E_\mu=(A_Z,\Id)\in SU(n-1)\times SU(2)\subset SO(2n+2)$. Then 
$E_\mu$ is an isometry on $(S^{2n+1},g_0)$ and a straightforward calculation shows that with $\kappa$, $\kappa^\prime$ associated 
with $j,j^\prime$ via \eqref{equation:kappa} we have $\mu\circ\kappa=E_\mu^*(\mu\circ\kappa^\prime)$ (compare the proof of 
\cite[Proposition 3.2.5]{MR1895349}).

Note that $E_\mu$ is $S^1$-equivariant and $\widetilde{T}=S^1\times S^1$-equivariant, hence induces a $T$-equivariant isometry $F_\mu$ on $(\O,g_0)$ and for any vector $X$ tangent to $\widehat{S^{2n+1}}$ we have $(\mu\circ\lambda)(P_*X)=F_\mu^*(\mu\circ{\lambda^\prime})(P_*X)$.
Since $P_{|\widehat{S^{2n+1}}}\co \widehat{S^{2n+1}}\to\widehat{\O}$ is a manifold submersion, $F_\mu$ satisfies condition \eqref{equation:Fmu} of Theorem \ref{theorem:isospect} on $\widehat{\O}$. Since both sides of \eqref{equation:Fmu} are smooth, it is satisfied on all of $\O$.
\end{proof}

We will show in Section \ref{subsection:nonisometry} that if $j,j^\prime$ are not equivalent and at least one of them is generic, then $(\O,g_\lambda)$ and $(\O,g_{\lambda^\prime})$ are not isometric.

Moreover, since $\langle,\rangle$ on $S^{2n+1}$ has constant curvature one and our quotient map $P\co (S^{2n+1},\langle,\rangle)\to(\O,g_0)$
is a Riemannian submersion, O'Neill's curvature formula (\cite{MR0200865}) implies that after multiplying $j$ and $j^\prime$ with a sufficiently small positive real number we can assume that the metrics $g^\reg_{\lambda},g^\reg_{\lambda^\prime}$ on $\O^\reg$ are so close to $g^\reg_0$ that they have positive curvature. Therefore $(\O,g_\lambda)$, $(\O,g_{\lambda^\prime})$ cannot be nontrivial Riemannian product orbifolds and they are not of the trivial form described at the beginning of Section \ref{section:examples}.

\subsubsection{Isospectral Families}
\label{subsubsection:examples-families}
The isospectrality proof for the pair $(\O,g_\lambda)$, $(\O,g_{\lambda^\prime})$ becomes considerably simpler if $j,j^\prime$ 
belong to a continuous isospectral family $j(t)$, $t\in I$. In this setting we can alternatively apply 
Theorem~\ref{theorem:isospect} (or \cite[Theorem 1.6]{MR1895349}) directly to the sphere (with $\langle,\rangle$ replaced by a 
nonstandard metric $h_0$) to deduce that the induced metrics on the quotient are isospectral. To this end we modify 
$\langle,\rangle$ in such a way that the fibres of our $S^1$-action \eqref{equation:defwp} become totally geodesic.

Use the standard metric $\langle,\rangle$ on $S^{2n+1}$ to define a new metric $h_0$ on $S^{2n+1}$ by setting for $(u,v)\in 
S^{2n+1}$, $X,Y\in T_{(u,v)}S^{2n+1}$:
\[h_0(X,Y):=(p^2\|u\|^2+q^2\|v\|^2)^{-1} \langle X^v,Y^v\rangle + \langle X^h,Y^h\rangle,\]
where the superscripts $v$ and $h$ refer to the vertical and horizontal parts with respect to the given $S^1$-action 
\eqref{equation:defwp} on $(S^{2n+1},\langle,\rangle)$. Note that this amounts to a smooth rescaling in the vertical directions; in 
particular, the horizontal spaces are the same for $\langle,\rangle$ and $h_0$ (as are the vertical spaces, of course).

Moreover, note that the action of $\widetilde{T}$ on $S^{2n+1}$ is still isometric with respect to $h_0$. Recall from Proposition \ref{proposition:kappalambda}  that if $j\co \mathfrak{t}\simeq\R^2\to \mathfrak{su}(n-1)$ is a linear map then the associated $\mathfrak{t}$-valued $1$-form $\kappa$, defined as in \eqref{equation:kappa} is
 $\widetilde{T}$-invariant and $\widetilde{T}$-horizontal, hence
admissible with respect to the $\widetilde{T}$-action on $S^{2n+1}$. For such $\kappa$ define $h_\kappa(X,Y):=h_0(X+\kappa(X)^*,Y+\kappa(Y)^*)$. In analogy to \cite[Proposition 3.2.5]{MR1895349} (but now with the deformed metric $h_0$ instead of $\langle,\rangle$) one then has:

\begin{proposition}
\label{proposition:isospheres}
If $j, j^\prime\co {\mathfrak t}\simeq \R^2\to \mathfrak{su}(n-1)$ are isospectral in the sense of Definition \rmref{definition:isomaps}\eqref{definition:isospectral} and $\kappa,\kappa^\prime$ are the corresponding ${\mathfrak t}$-valued $1$-forms on $S^{2n+1}$ given by \eqref{equation:kappa}, then $(S^{2n+1},h_\kappa)$ and $(S^{2n+1},h_{\kappa^\prime})$ are isospectral manifolds.
\end{proposition}

\begin{proof}
  We had already recalled above how the isospectrality condition was used in \cite[Proposition 3.2.5]{MR1895349} to find for each 
  $\mu\in{\mathcal L}^*$ an isometry 
\[E_\mu=(A,\Id)\in SU(n-1)\times SU(2)\subset SO(2n+2)\] 
on $(S^{2n+1},\langle,\rangle)$ satisfying $\mu\circ\kappa=E_\mu^*(\mu\circ\kappa^\prime)$. Note that $E_\mu$ acts isometrically on 
$(S^{2n+1},h_0)$ as well. The proposition then follows from Theorem \ref{theorem:isospect} (or from \cite[Theorem 1.6]{MR1895349}).
\end{proof}

\begin{remark}
Using Proposition \ref{proposition:kappalambda} and the denseness of $\widehat{S^{2n+1}}$ in $S^{2n+1}$, it is not hard to see that given a fixed $j$, the induced metric $h_\kappa^{S^1}$ on our orbifold $\O=S^{2n+1}/S^1$ coincides with the metric $g_\lambda$ from \ref{notations:O-T}(\ref{enumremark:glambda}).
\end{remark}

\begin{proposition}
\label{proposition:specincl}
$\spec(\O,h_\kappa^{S^1})\subset \spec(S^{2n+1},h_\kappa)$.
\end{proposition}

\begin{proof}
With respect to the metric $h_0$ all regular $S^1$-orbits are easily seen to have length $2\pi$.
Since $\kappa$ is $S^1$-horizontal, we obtain the same result with respect to $h_\kappa$ 
and hence the Riemannian manifold submersion \[P\co (\widehat{S^{2n+1}},h_\kappa)\to(\widehat{S^{2n+1}}/S^1,h_\kappa^{S^1})\] has 
totally geodesic fibres. Together with the denseness of $\widehat{S^{2n+1}}$ in $S^{2n+1}$ this implies that every eigenfunction on 
$\O$ pulls back to an eigenfunction on $S^{2n+1}$ associated with the same eigenvalue. Since linear independence is also preserved 
by $P^*$, the proposition follows.
\end{proof}

\begin{remark}
For the spectrum in the setting of Riemannian orbifold submersions with totally geodesic fibres also compare \cite{MR2203544}.
\end{remark}

Finally, we obtain the following proposition, which is actually just a special case of Theorem \ref{theorem:isoorbis}.

\begin{proposition}
Given a continuous isospectral family of linear maps $j(t)\co  \mathfrak{t}\to\mathfrak{su}(n-1),t\in I$, the associated Riemannian metrics $h_{\kappa(t)}^{S^1}=g_{\lambda(t)}$ on $\O=\O(p,q)=S^{2n+1}/S^1$ form a continuous family of isospectral metrics on the orbifold $\O$.
\end{proposition}

\begin{proof}
We write $0=\mu_0(t)<\mu_1(t)\le\mu_2(t)\le\dotsb$ for the spectrum of
$(\O,h_{\kappa(t)}^{S^1})$ and note that
each of these functions $\mu_i\co I\to[0,\infty)$ is continuous (as can be seen as in the compact manifold setting using Theorem 
\ref{theorem:variation}). From Proposition \ref{proposition:specincl} in connection with Proposition \ref{proposition:isospheres} 
we deduce that the image of each $\mu_i$ is discrete. Since $I$ is connected, this implies that each $\mu_i$ is constant.
\end{proof}

\subsection{Nonisometry}
\label{subsection:nonisometry}
In this section we will argue that if $j,j^\prime$ are not equivalent and at least one of them is generic in the sense of Definition \ref{definition:isomaps}, then the corresponding metrics $g_\lambda$ and $g_{\lambda^\prime}$ on $\O=\O(p,q)=S^{2n+1}/S^1$ are not isometric. This will be a consequence of Proposition \ref{proposition:example-NG} and the nonisometry criterion in Section \ref{subsection:general-nonisometry}. Together with the results from Subsection \ref{subsection:isometrics} we will finally obtain the main result of this paper (Theorem \ref{theorem:maintheorem}).

Some of the arguments below are inspired by ideas in \cite{rueckriemen06}. (However, we do not use concrete results due to a 
mistake in \cite[Remark 5.9]{rueckriemen06}, compare the first step in the proof of Proposition \ref{proposition:example-NG}.) 
Before we can use the criterion from Proposition \ref{proposition:nonisometry}, we need some preliminary observations. 
Proposition~\ref{proposition:example-NG} will then be a consequence of Lemmas \ref{lemma:Oab} and \ref{lemma:DinE}. As usual, we 
will use the canonical metrics and the corresponding submersion metrics unless otherwise stated.

Let $\widetilde{T}=(S^1)^2$ act on $\C^{n-1}\setminus\{0\}\times(\C^*)^2$ by multiplication in the last two components and consider the following isometric $S^1$-actions, where $\sigma\in S^1\subset \C$, $u\in\C^{n-1}\setminus\{0\}$, $v\in(\C^*)^2$, $a,b\in\R_{>0}$:
\begin{itemize}
 \item On $(\C^{n-1}\setminus\{0\}\times(\C^*)^2)/\widetilde{T}$ set $\sigma\qak{(u,v)}:=\qak{(\sigma^pu,\sigma^qv)}$.
 \item On $\C^{n-1}\setminus\{0\}\times\R_{>0}\times\R_{>0}$ set $\sigma(u,a,b):=(\sigma^pu,a,b)$.
\end{itemize}
With respect to these actions, the isometry
\[(\C^{n-1}\setminus\{0\}\times(\C^*)^2)/\widetilde{T}\ni\qak{(u,v)}\mapsto(u,|v_1|,|v_2|)\in \C^{n-1}\setminus\{0\}\times\R_{>0}\times\R_{>0}\]
is $S^1$-equivariant.
Now recall from Section \ref{subsubsection:isospectral-pairs} that
\[\widehat{S^{2n+1}}=\{(u,v)\in\C^{n-1}\times\C^2;~\|u\|^2+\|v\|^2=1, u\ne 0, v_j\ne 0~\forall j=1,2\}\]
and restrict the $S^1$-equivariant isometry above to the $S^1$-invariant submanifold $\widehat{S^{2n+1}}/\widetilde{T}$ of $(\C^{n-1}\setminus\{0\}\times(\C^*)^2)/\widetilde{T}$. Factoring out the $S^1$-actions gives an isometry
\begin{equation}
  \label{equation:phi}
  \Phi\co \widehat{\O}/\widetilde{T}\to N/S^1,
\end{equation}
where $N:=\{(u,a,b)\in\C^{n-1}\setminus\{0\}\times \R_{>0}\times\R_{>0};~\|u\|^2+a^2+b^2=1\}\subset S^{2n-1}$.

Note that the $S^1$-actions above are not effective. However, the quotient of $S^1$ by the $p$-roots of unity acts freely and it is the smooth structures induced by these free actions that we refer to. Analogously, $\widehat{\O}/\widetilde{T}$ is just $\widehat{\O}/T$ with $T=\widetilde{T}/\{(\sigma,\sigma)\in \widetilde{T};~\sigma^p=1\}$ acting freely on $\widehat{\O}$.

Recall from Section \ref{subsection:general-nonisometry} that $\pi\co \widehat{\O}\to\widehat{\O}/T$ denotes the quotient map. For $a,b>0$ with $a^2+b^2<1$ set
\begin{align*}
S_{a,b}&:=(S^{2n-3}(\sqrt{1-a^2-b^2})\times\{(a,b)\})/S^1\subset N/S^1,\\
\O_{a,b}&:=\pi^{-1}(\Phi^{-1}(S_{a,b}))\subset \widehat{\O}.
\end{align*}
Since $\pi$ is a manifold submersion, $\O_{a,b}$ is a $T$-invariant submanifold of $\widehat{\O}$. By definition, under the isometry $\Phi$ the manifold $\O_{a,b}/T$ corresponds to 
$S_{a,b}\stackrel{\text{isom.}}{\simeq}(\CP^{n-2},(1-a^2-b^2)g_\text{FS})$,
where $g_\text{FS}$ denotes the Fubini--Study metric on $\CP^{n-2}$.

For $x\in \widehat{S^{2n+1}}$ consider the diffeomorphism 
\[r^x\co \widetilde{T}=S^1\times S^1\ni(\sigma_1,\sigma_2)\mapsto(\sigma_1,\sigma_2) x\in \widetilde{T}x\subset \widehat{S^{2n+1}}\]
and the corresponding immersion 
\[r^\qao{x}\co \widetilde{T}=S^1\times S^1\ni(\sigma_1,\sigma_2)\mapsto(\sigma_1,\sigma_2) \qao{x}\in \widetilde{T}\qao{x}\subset \widehat{S^{2n+1}}/S^1=\widehat{\O}.\]
Note that $r^\qao{x}=P\circ r^x$ for $P\co \widehat{S^{2n+1}}\to\widehat{S^{2n+1}}/S^1=\widehat{\O}$ the canonical projection. In the following calculations we will use our convention that on $\O$ the bracket $\langle,\rangle$ stands for $g_0$. A straightforward calculation shows:
\begin{proposition}
\label{proposition:skalprodukt}
Let $A_j, B_j\in\R$, $\sigma_j\in S^1\subset\C$ for $j=1,2$, and set 
\[A:=(iA_1\sigma_1,iA_2\sigma_2), B:=(iB_1\sigma_1,iB_2\sigma_2)\in T_{(\sigma_1,\sigma_2)}(S^1\times S^1)\subset\C^2.\]
Moreover, let $x=(u,v)\in \widehat{S^{2n+1}}$ with $u\in\C^{n-1}$, $v\in\C^2$. Then
\[\langle r_*^\qao{x}A,r_*^\qao{x}B\rangle=\sum_{j=1}^2 A_jB_j|v_j|^2-\frac{q^2(\sum_j A_j|v_j|^2)(\sum_j B_j|v_j|^2)}{p^2\|u\|^2+q^2\|v\|^2}.\]
\end{proposition}

Recall that $Z_1=(i,0), Z_2=(0,i)$ denote the standard basis of $\mathfrak{t}=T_{(1,1)}(S^1\times S^1)\subset\C^2$ and that for $Z\in\mathfrak{t}$ the symbol $\widehat{Z}$ denotes the fundamental manifold vector field associated with $Z$ with respect to the action of $T$ (or, equivalently, $\widetilde{T}$) on $\widehat{\O}$. Moreover, note that Proposition \ref{proposition:kappalambda}(\ref{proposition:kappalambda:1}) implies $\widehat{Z_k}\circ P=P_*\circ Z_k^*=P_*{r^\cdot_*}_{(1,1)}Z_k={r^{[\cdot]}_*}_{(1,1)}Z_k$.

\begin{corollary}
\label{corollary:skalprod}
For $j,k\in\{1,2\}$ and $x=(u,v)\in \widehat{S^{2n+1}}$ we have
\[\langle {\widehat{Z_j}}_\qao{x},{\widehat{Z_k}}_\qao{x}\rangle=\delta_{jk}|v_j|^2-\frac{q^2|v_j|^2|v_k|^2}{p^2\|u\|^2+q^2\|v\|^2}.\]
\end{corollary}
\begin{proof}
Apply Proposition \ref{proposition:skalprodukt} to ${\widehat{Z_1}}_{\qao{x}}={r_*^{\qao{x}}}Z_1$ and ${\widehat{Z_2}}_{\qao{x}}={r_*^{\qao{x}}}Z_2$ in $\sigma=(1,1)$.
\end{proof}

\begin{corollary}
\label{corollary:winkel}
For $\qao{x}\in \O_{a,a}$ we have:
\begin{align}
\label{corollary:winkel:1} \langle {\widehat{Z_j}}_{\qao{x}},{\widehat{Z_k}}_{\qao{x}}\rangle&=\delta_{jk}a^2-\frac{q^2a^4}{p^2(1-2a^2)+2q^2a^2},
\\
\label{corollary:winkel:2} \angle({\widehat{Z_1}}_{\qao{x}},{\widehat{Z_2}}_{\qao{x}})&=\arccos\frac{-q^2a^2}{p^2(1-2a^2)+q^2a^2}.
\end{align}
\end{corollary}
\begin{proof}
\eqref{corollary:winkel:1} follows directly from Corollary \ref{corollary:skalprod}. \eqref{corollary:winkel:2}  then follows from \eqref{corollary:winkel:1}.
\end{proof}

Moreover, since $T$ is abelian and acts by isometries, we can make the following observation.

\begin{lemma}
\label{lemma:torus-metric}
Given $a,b>0$ with $a^2+b^2<1$ and $\qao{x}\in\O_{a,b}$, the map $R^\qao{x}\co T\ni z\mapsto z\qao{x}\in \widehat{\O}$ is an embedding and the pull-back by $R^\qao{x}$ of the metric $g_0=\langle,\rangle$ to $T$ is left-invariant and associated with the inner product
$\mathfrak{t}^2\ni(Y_1,Y_2)\mapsto \left\langle \widehat{Y_1}_\qao{x},\widehat{Y_2}_\qao{x}\right\rangle\in\R$.
\end{lemma} 

We first use the formulas above to show the following lemma, from which we will need only the case $a=b$ in the proof of Proposition \ref{lemma:DinE}. Recall from \ref{proposition:nonisolist}(\ref{proposition:nonisolist-Aut}) that $\Aut^T_{g_0}(\O^\reg)$ is the group of all $T$-preserving diffeomorphisms of $\O^\reg$ which preserve the $g_0$-norm of vectors tangent to the $T$-orbits in $\widehat{\O}$ and induce an isometry of $(\widehat{\O}/T,g_0^T)$.

\begin{lemma}
\label{lemma:Oab}
Let $a,b>0$ with $a^2+b^2<1$ and $F\in \Aut^T_{g_0}(\O^\reg)$. Then $F(\O_{a,b}\cup \O_{b,a})= \O_{a,b}\cup \O_{b,a}$.
\end{lemma}

\begin{proof}
For $c\in (0,1)$ set
\[\O_c=\bigcup_{\substack{r^2+s^2=1-c^2\\ r,s>0}} \O_{r,s}\subset \widehat{\O}.\]
We proceed in two steps.

\subsubsection*{First step}
We will first show that $F$ preserves every $\O_c$. To this end for each $c\in (0,1)$ set
\[N_c:=S^{2n-3}(c)\times\{(r,s)\in(\R_{>0})^2;~r^2+s^2=1-c^2\}\subset S^{2n-1}\subset \C^{n-1}\times\R^2\]
and observe that $N_c/S^1=\Phi(\O_c/T)$ and $N=\bigcup_{c\in(0,1)}N_c$ (with $\Phi$ and $N$ given in \eqref{equation:phi}).

Now fix $c\in (0,1)$. Note that $\O_c$ is $T$-invariant and hence $F$ preserves $\O_c$ if and only if the induced isometry $\ol{F}\in\ol{\Aut}^T_{g_0}(\O^\reg)$ of $\widehat{\O}/T$ leaves $\O_c/T$ invariant.
The isometry $\Phi\co \widehat{\O}/T\to N/S^1$ has a unique continuous extension
\[\widetilde{\Phi}\co \O/T=\widetilde{\widehat{\O}}/T=\widetilde{\widehat{\O}/T}\to \widetilde{N/S^1}=\widetilde{N}/S^1=\ol{N}/S^1,\]
where the tildes denote the completions of the respective metric spaces. This extension is given by $(S^{2n+1}/S^1)/T\ni[(u,v_1,v_2)]\mapsto [(u,|v_1|,|v_2|)]$. Write $\widetilde{\pi}\co \O\to\O/T$ for the canonical projection and note that $\widetilde{\pi}$ is the unique continuous extension of $\pi\co \widehat{\O}\to\widehat{\O}/T$. Moreover, note that
\[\ol{N}=\{(u,r,s)\in \C^{n-1}\times\R_{\ge 0}\times\R_{\ge 0};\|u\|^2+r^2+s^2=1\}\subset S^{2n-1}\subset\C^{n-1}\times\R^2.\]
Extend $\bar{F}\in\Isom(\widehat{\O}/T)$ uniquely to a metric space isometry $\widetilde{F}$ of $\O/T$ and note that $\widetilde{F}\circ\widetilde{\pi}=\widetilde{\pi}\circ F$ by continuity. Now set $N_1:=S^{2n-3}\times\{(0,0)\}\subset\ol{N}$. Then it is straightforward to see:
\begin{enumerate}
\item \label{enumerate:N1} $N_1/S^1$ is preserved by the isometry $\widetilde{\Phi}\circ\widetilde{F}\circ\widetilde{\Phi}^{-1}$.
\item \label{enumerate:N2} $N_c/S^1$ is precisely the set of all points in $\widetilde{N}/S^1$ which have distance $\arccos(c)$ from $N_1/S^1$.
\end{enumerate}
(\ref{enumerate:N1}) and (\ref{enumerate:N2}) together imply that $\widetilde{\Phi}\circ\widetilde{F}\circ\widetilde{\Phi}^{-1}\in\Isom(\ol{N}/S^1)$ preserves $N_c/S^1$, hence $\ol{F}=\widetilde{F}_{|\widehat{\O}/T}$ leaves $\O_c/T$ invariant and $\O_c$ is preserved by $F$.

\subsubsection*{Second step} Now let $a,b>0$ with $a^2+b^2<1$ and $F\in \Aut^T_{g_0}(\O^\reg)$ be as in the lemma and fix $\qao{x}=\qao{(u,v)}\in\O_{a,b}$.
Note that the area of $T$  with respect to its standard bi-invariant metric (with $\{Z_1,Z_2\}$ an orthonormal basis of $\mathfrak{t}$) is $4\pi^2/p$, since $T\simeq \mathfrak{t}/{\mathcal L}$ with ${\mathcal L}=\spann_\Z\left\{2\pi Z_1,\frac{2\pi}{p}(Z_1+Z_2)\right\}$.
By Lemma \ref{lemma:torus-metric} we conclude that the area of $T\qao{x}$ is given by $A(T\qao{x})=\frac{4\pi^2}{p}\sqrt{\det(\langle {\widehat{Z_j}}_{\qao{x}},{\widehat{Z_k}}_{\qao{x}}\rangle)_{j,k=1,2}}$. Now set $c=\|u\|=\sqrt{1-a^2-b^2}$ so that $\O_{a,b}\cup\O_{b,a}\subset\O_c$. Corollary \ref{corollary:skalprod} then implies
\begin{equation}
  \label{equation:area}
  \frac{p^2}{16\pi^4}A(T\qao{x})^2=a^2b^2\left(1-\frac{q^2(1-c^2)}{p^2c^2+q^2(1-c^2)}\right).
\end{equation}

Note that since $F$ preserves the length of vectors tangent to $T$-orbits by definition, we have $A(T\qao{x})=A(F(T\qao{x}))=A(T F(\qao{x}))$. Moreover, we had seen in the first step that $\O_c$ is invariant under $F$. These two observations and \eqref{equation:area} then imply that for $[(u^\prime,v_1^\prime,v_2^\prime)]:=F(\qao{x})$ and $a^\prime:=|v_1^\prime|,b^\prime:=|v_2^\prime|$, we have ${a^\prime}^2+{b^\prime}^2=a^2+b^2$ and ${a^\prime}^2{b^\prime}^2=a^2b^2$, i.e., $a=a^\prime~\wedge~b=b^\prime$ or $a=b^\prime~\wedge~b=a^\prime$. In other words, $F$ preserves $\O_{a,b}\cup \O_{b,a}$. Since $F^{-1}$ also lies in $\Aut^T_{g_0}(\O^\reg)$, the lemma follows.
\end{proof}

Recall that we had set ${\mathcal D} := \{\Psi_F;~F\in\Aut^T_{g_0}(\O^\reg)\}\subset\Aut({\mathfrak t})$ in \ref{proposition:nonisolist}(\ref{proposition:nonisolist-D}) and ${\mathcal E}:=\{\phi\in\Aut({\mathfrak t});~\phi(Z_k)\in\{\pm Z_1,\pm Z_2\}~\forall k=1,2\}$ in Definition \ref{definition:isomaps}(\ref{definition:equivalent}). We are now in a position to show that in our example we have the following inclusion.

\begin{lemma}
\label{lemma:DinE}
${\mathcal D}\subset{\mathcal E}$.
\end{lemma}

\begin{proof}
Let $F\in\Aut_{g_0}^T(\O^\reg)$.
We have to show that $\Psi_F(Z_k)\in\{\pm Z_1,\pm Z_2\}$ for $k=1,2$. By \ref{proposition:nonisolist}(\ref{proposition:nonisolisti}) 
we know $F_*(\widehat{Z_k})=\widehat{\Psi_F(Z_k)}$ on $\widehat{\O}$. The map \[\mathfrak{t}\ni Z\mapsto \widehat{Z}_\qao{x}\in 
T_\qao{x}\widehat{\O}\] is injective for any $[x]\in\widehat{\O}$ as the differential of the embedding $R^{[x]}:T\ni z\mapsto 
z[x]\in\widehat{\O}$. So it suffices to show that 
\[{F_*}_\qao{x}(\widehat{Z_k}_\qao{x})\in\bigl\{\pm\widehat{Z_1}_{F(\qao{x})},\pm\widehat{Z_2}_{F(\qao{x})}\bigr\}\] 
for $k=1,2$ in a single 
point $\qao{x}\in\widehat{\O}$.

By \eqref{corollary:winkel:2} we can choose $a\in (0,\frac1{\sqrt{2}})$ such that 
$\mathrm{cos}\angle(\widehat{Z_1}_\qao{x},\widehat{Z_2}_\qao{x})$ is irrational for all $\qao{x}\in\O_{a,a}$.
Now fix $\qao{x}\in \O_{a,a}$ and temporarily write $\langle Y_1, 
Y_2\rangle:=\langle\widehat{Y_1}_\qao{x},\widehat{Y_2}_\qao{x}\rangle$ and $\|Y\|:=\sqrt{\langle Y,Y\rangle}$ for 
$Y_1,Y_2,Y\in\mathfrak{t}$. Note that $\|Z_1\|=\|Z_2\|$.
By our choice of $a$ we observe that if $k,l\in\Z$ and \[\|kZ_1+lZ_2\|^2/\|Z_1\|^2\in\Q,\] then $kl=0$.
Hence if $2\pi Y\in{\mathcal L}$ with $\|Y\|=\|Z_1\|$, then (since $2\pi pY\in\widetilde{\mathcal L}$ and hence $pY=k Z_1+l Z_2$ for 
some $k,l\in\Z$) we have $pY\in\{\pm pZ_1,\pm pZ_2\}$ and hence $Y\in\{\pm Z_1,\pm Z_2\}$.

This implies that the images of the two flow lines generated by $\widehat{Z_1}$ and $\widehat{Z_2}$ through $[x]$ give precisely the 
geodesic loops in $T\qao{x}\subset \O_{a,a}$ through $[x]$ of length $2\pi\|Z_1\|$; recall from Lemma \ref{lemma:torus-metric} that 
$T\ni z\mapsto z[x]\in T[x]$ is an isometry with respect to the left-invariant metric $\langle,\rangle$ above on $T$, hence such 
flow lines are indeed geodesics.

Since $F$ preserves $\O_{a,a}$ by Lemma \ref{lemma:Oab}, we have $F(T[x])\subset\O_{a,a}$ and the geodesic loops in 
$TF([x])=F(T[x])$ through $F([x])$ of length $2\pi\|Z_1\|$ are given precisely by 
the flow lines of $\widehat{Z_1}$ and $\widehat{Z_2}$ through $F([x])$.
On the other hand, since $F\co T\qao{x}\to F(T[x])$ is an isometry, the images of the flow lines of $F_*\widehat{Z_1}$ and 
$F_*\widehat{Z_2}$ through $F([x])$ in $F(T[x])$ also have length $2\pi\|Z_1\|$. Together this implies 
${F_*}_{\qao{x}}({\widehat{Z_j}}_{[x]})\in\{\pm\widehat{Z_1}_{F(\qao{x})},\pm\widehat{Z_2}_{F(\qao{x})}\}$ for $j=1,2$. As noted 
above, this proves our statement.
\end{proof}

Proposition \ref{proposition:nonisometry} (where we introduced the properties \eqref{eq:N} and \eqref{eqG} below) and the following proposition will 
imply that if isospectral maps $j$ and $j^\prime$ are not equivalent and $j^\prime$ is generic, then the corresponding isospectral 
orbifolds $(\O,g_\lambda)$, $(\O,g_{\lambda^\prime})$ with $\O=\O(p,q)$ are nonisometric. In the second and third step of the proof 
we basically follow the first part of the proof of \cite[Proposition 4.3]{MR1895349}. 

\begin{proposition}
  \label{proposition:example-NG}
Let $j,j^\prime\co \R^2\to \mathfrak{su}(n-1)$ be two linear maps and let $\lambda$, $\lambda^\prime$ be the admissible $\mathfrak{t}$-valued $1$-forms on $\O=\O(p,q)$ associated with $j$ and $j^\prime$.
\begin{enumerate}
  \item \label{enumerate:wp-N} If $j$ and $j^\prime$ are not equivalent in the sense of Definition \rmref{definition:isomaps}\eqref{definition:equivalent}, then $\Omega_\lambda$ and $\Omega_{\lambda^\prime}$ satisfy condition \eqref{eq:N}.
  \item\label{4.14.ii} 
  If $j^\prime$ is generic in the sense of Definition \rmref{definition:isomaps}\eqref{definition:generic}, then 
  $\Omega_{\lambda^\prime}$ has property \eqref{eqG}.
\end{enumerate}
\end{proposition}

\begin{proof}
Choose an arbitrary $a\in (0,1/\sqrt{2})$ and set $L:=\O_{a,a}\subset \widehat{\O}$. We write $\Omega_0^L$ for the $\mathfrak{t}$-valued 2-form on $L/T$ induced by the curvature form $\Omega_0$ on $(\widehat{\O}/T,g_0^T)$. Moreover, to a $\mathfrak{t}$-valued $k$-form $\eta$ on a manifold we associate real-valued $k$-forms $\eta^1,\eta^2$ via $\eta=:\eta^1 Z_1+\eta^2 Z_2$.

\subsubsection*{First step: calculation of $\Omega_0^L$} In this step we will show that on \[L/T\stackrel{\text{isom}}{\simeq}(\CP^{n-2},(1-2a^2)g_\text{FS})\] we have $(\Omega_0^L)^1=(\Omega_0^L)^2$ and this form is a nonvanishing multiple of the standard K\"ahler form.

Recall from \ref{proposition:nonisolist}(\ref{propositon:nonisolist-iv}) that $\omega_0\co  T\widehat{\O}\to\mathfrak{t}$ denotes the connection form on the principal $T$-bundle $\widehat{\O}$ associated with $g_0$. We first note that it is not hard to verify that with $P\co \widehat{S^{2n+1}}\to\widehat{\O}$ the canonical projection we have for $(u,v)\in \widehat{S^{2n+1}}$, $X=(U,V)\in T_{(u,v)}\widehat{S^{2n+1}}$, $j=1,2$:
\begin{equation}
 \label{equation:omega0}
 (P^*{\omega_0}^j)_{(u,v)}(X)=-\frac{q}{p}\frac{\langle U,iu\rangle}{\|u\|^2}+\frac{\langle V_j,iv_j\rangle}{|v_j|^2}
\end{equation}
Now write $\omega_0^L$ for the $\mathfrak{t}$-valued $1$-form on $L$ induced by $\omega_0$. \eqref{equation:omega0} implies that for $(u,v)\in P^{-1}(L)$ and $X=(U,V)\in T_{(u,v)}P^{-1}(L)$:
\begin{equation}
  \label{equation:omega0l}
  (P^*\omega_0^L)^j(X)=-\frac{q}{p(1-2a^2)}\langle U,iu\rangle+\frac1{a^2}{\langle V_j,i v_j\rangle}.
\end{equation}
Now note that $P^{-1}(L)=S^{2n-3}(\sqrt{1-2a^2})\times (S^1(a))^2\subset S^{2n+1}\subset\C^{n+1}$ and hence if $X=(U,V)\in T_{(u,v)}P^{-1}(L)$, then $V_j$ is a real multiple of $iv_j$ for $j=1,2$. Using this, \eqref{equation:omega0l} implies that for $X=(U,V), \widetilde{X}=(\widetilde{U},\widetilde{V})$ tangent to $P^{-1}(L)$ in $(u,v)$ and $j=1,2$:
\[ (P^*d\omega_0^L)^j(X,\widetilde{X})=-2\frac{q}{p(1-2a^2)}\langle iU,\widetilde{U}\rangle. \]
Therefore, on $L/T\stackrel{\text{isom}}{\simeq}(\CP^{n-2},(1-2a^2)g_\text{FS})$ the form $(\Omega_0^L)^1=(\Omega_0^L)^2$ is a nonvanishing multiple of the standard K\"ahler form.

\subsubsection*{Second step: proof of \eqref{enumerate:wp-N}} Suppose that condition \eqref{eq:N} is not satisfied. Then there is $\Psi\in{\mathcal D}$ and $F\in\Aut_{g_0}^T(\O^\reg)$ such that $\Omega_\lambda=\Psi\circ\qbo{F}^*\Omega_{\lambda^\prime}$. Since $\qbo{F}$ preserves $L/T$ by Lemma \ref{lemma:Oab}, this implies $\Omega_\lambda^L=\Psi\circ\qbo{F}^*\Omega_{\lambda^\prime}^L$. Now $\Omega_\lambda=\Omega_0+d\qbo{\lambda}$ and $\Omega_{\lambda^\prime}=\Omega_0+d\qbo{\lambda^\prime}$ (\ref{proposition:nonisolist}(\ref{proposition:nonisolist-vi})) imply (with $\ol{\lambda}^L$ denoting the $\mathfrak{t}$-valued $1$-form on $L/T$ induced by $\ol{\lambda}$, and analogously for $\ol{\lambda^\prime}$):
\begin{equation}
 \label{equation:Omega-lambda}
 \Omega_0^L+d\ol{\lambda}^L=\Omega_\lambda^L=\Psi\circ\ol{F}^*\Omega_{\lambda^\prime}^L=\Psi\circ\ol{F}^*(\Omega_0^L+d\ol{\lambda^\prime}^L).
\end{equation}
In particular, $\Omega_0^L-\Psi\circ\ol{F}^*\Omega_0^L$ is exact. Moreover, note that Proposition \ref{lemma:DinE} implies $\Psi\in{\mathcal E}$. The first step above shows that $\Omega_0^L-\Psi\circ\ol{F}^*\Omega_0^L\in\{0,2\Omega_0^L\}$ and that $2\Omega_0^L$ cannot be exact. Hence $\Omega_0^L-\Psi\circ\ol{F}^*\Omega_0^L=0$ and \eqref{equation:Omega-lambda} implies
\begin{equation}
\label{equation:dlambda}
d\qbo{\lambda}^L=\Psi\circ\qbo{F}^*d\qbo{\lambda^\prime}^L.
\end{equation}
Let $Q\co \C^{n-1}\to\C^{n-1}$ denote complex conjugation and choose \[A\in SU(n-1)\cup SU(n-1)\circ Q\] such that $A$ induces (via the Hopf fibration $\C^{n-1}\supset S^{2n-3}\to \CP^{n-2}$) the isometry on $L/T\simeq (\CP^{n-2},(1-2a^2)g_\text{FS})$ corresponding to $\qbo{F}_{|L/T}$, i.e., such that $P\circ(A,I_2)_{|P^{-1}(L)}=F\circ P_{|P^{-1}(L)}$. Then, with $\kappa^L$ denoting the restriction of $\kappa$ to $P^{-1}(L)$ (and analogously for $\kappa^\prime$), pulling back both sides of \eqref{equation:dlambda} via $\pi\circ P$, we obtain
\begin{equation}
d\kappa^L=\Psi\circ (A,I_2)^*d{\kappa^\prime}^L.\label{equation:dkappa}
\end{equation}
For $k\in\{1,2\}$ set $j_k:=j_{Z_k}$. Letting \[(u,v)\in P^{-1}(L)=S^{2n-3}(\sqrt{1-2a^2})\times(S^1(a))^2,\] we have by 
\eqref{equation:kappa} for $(U,V)\in T_{(u,v)}(P^{-1}(L))$:
\[\kappa^k_{(u,v)}({U},{V})=(1-2a^2)\langle j_k u,{U}\rangle-\langle {U},iu\rangle\langle j_ku,i u\rangle.\]
For $(U_1,V_1),(U_2,V_2)\in T_{(u,v)}(P^{-1}(L))$ we get by elementary differentiation and skew-symmetry:
\begin{align}\label{equation:dkappa-u1u2}
d\kappa_{(u,v)}^k&(({U_1},{V_1}),({U_2},{V_2}))
\\&=2(1-2a^2)\langle j_kU_1,U_2 \rangle-2\langle iU_1,U_2\rangle\langle j_ku,iu\rangle\notag\\
&\quad\;-2\langle U_2,iu\rangle\langle j_kU_1,iu\rangle+2\langle U_1,iu\rangle\langle j_kU_2,iu\rangle\notag\\
&=2(1-2a^2)\langle j_kU_1^h,U_2^h\rangle -2\langle j_ku,iu\rangle \langle i{U_1},{U_2}\rangle,
\notag\end{align}
where we write $U^h=U-\frac{\langle U,iu\rangle}{1-2a^2}iu$ for the orthogonal projection of $U\in T_uS^{2n-3}(\sqrt{1-2a^2})$ to $(iu)^\perp$. By Proposition \ref{lemma:DinE} we can choose $\varepsilon_{k}\in\{\pm 1\}$ and $l\in\{1,2\}$ such that $\Psi(Z_k)=\varepsilon_kZ_l$. Plugging \eqref{equation:dkappa-u1u2} and the analogous formula for $\kappa^\prime$, $j^\prime$ into \eqref{equation:dkappa}, we obtain:
\begin{align*}
  2(1-2a^2)\langle &j_lU_1^h,U_2^h\rangle-2\langle j_lu,iu\rangle\langle i{U_1},{U_2}\rangle\\
&= 2\varepsilon_k\left((1-2a^2)\langle A^{-1}j_k^\prime 
  AU_1^h,U_2^h\rangle- \langle A^{-1}j_k^\prime Au,iu\rangle\langle i{U_1},{U_2}\rangle\right).
\end{align*}
Setting $\tau_k:=\varepsilon_kA^{-1}j_k^\prime A-j_l\in \mathfrak{su}(n-1)$ gives 
\[0=(1-2a^2)\langle \tau_kU_1^h,U_2^h\rangle-\langle \tau_ku,iu\rangle\langle i{U_1},{U_2}\rangle.\]
Plugging ${U_2}=i{U_1}$ into the last equation, we can conclude that
the map $\phi\co \C^{n-1}\setminus\{0\}\ni U\mapsto\frac{\langle i \tau_k{U},U\rangle}{\|{U}\|^2}\in\R$ is constant, say $C$, on $\spann\{u,iu\}^\perp\setminus\{0\}$ and $\phi(u)=\phi(iu)=C$. Since $i\tau_k$ is hermitian and has trace zero, we have $\tau_k=0$.
This finally implies $j_{\Psi(Z_k)}=A^{-1}j^\prime_{Z_k}A$ for $k=1,2$ and therefore $j_{\Psi(Z)}=A^{-1}j^\prime_ZA$.

\subsubsection*{Third step: proof of \eqref{4.14.ii}} Assume that $\Omega_{\lambda^\prime}$ does not satisfy property \eqref{eqG}. Then there is a nontrivial one-parameter family $\bar{F}_t\in\ol{\Aut}_{g_0}^T(\O^\reg)$ such that $\bar{F}_t^*\Omega_{\lambda^\prime}=\Omega_{\lambda^\prime}$ for all $t$. The same argument as above (with $\Psi=\Id$ and $j=j^\prime$) gives a one-parameter family $A_t\in SU(n-1)\cup SU(n-1)\circ Q$ such that $(A_t,I_2)$ preserves $d{\kappa^\prime}^L$. Note that $A_0=\Id$ implies $A_t\in SU(n-1)$. As in the proof of (\ref{enumerate:wp-N}) the relation $(A_t,I_2)^*d{\kappa^\prime}^L=d{\kappa^\prime}^L$ implies $j^\prime_Z=A_t j^\prime_ZA_t^{-1}$. Taking the derivative with respect to $t$ in $0$ gives $0=[\dot{A}_0,j^\prime_Z]$ for all $Z\in\mathfrak{t}$ in contradiction to the genericity assumption.
\end{proof}

As announced in the beginning of this section, we can now put all pieces together to obtain our main result.

\begin{theorem}
\label{theorem:maintheorem}
 For every $n\ge 4$ and for all pairs $(p,q)$ of coprime positive integers there are isospectral families of pairwise nonisometric metrics on the orbifold $\O=\O(p,q)$, a weighted projected space of dimension $2n\ge 8$, which is a bad orbifold for $(p,q)\ne (1,1)$.
\end{theorem}
\begin{proof}
  Direct consequence of Proposition \ref{proposition:isomaps}, Theorem \ref{theorem:isoorbis}, 
  Proposition~\ref{proposition:nonisometry}, and Proposition \ref{proposition:example-NG}.
\end{proof}

\subsection{Isospectral quotients of weighted projective spaces}
\label{subsection:sutton}
In this section we will apply ideas from \cite{sutton10} to give isospectral metrics on quotients of the form $\O(p,q)/G$ with 
$\O(p,q)$ from the preceding sections and $G$ now a finite subgroup of the given $2$-torus $T$ (which had been introduced in 
Section~\ref{subsubsection:isospectral-pairs}).

We first phrase a special case of Sutton's results on equivariant isospectrality (\cite{sutton10}, also compare \cite{MR1396675}) for orbifolds. Suppose we are given a Riemannian orbifold $(\O,g)$ and a finite subgroup $G$ of its isometry group. Then $\O/G$ carries a canonical orbifold structure and a Riemannian orbifold metric $\bar{g}$ such that the canonical projection $P\co (\O,g)\to(\O/G,\bar{g})$ becomes a Riemannian orbifold covering.

\begin{theorem}
\label{theorem:suttonorbi}
Let $G$ be a finite group acting effectively and isometrically on two compact Riemannian orbifolds $(\O_1,g_1)$ and $(\O_2,g_2)$ 
such that the latter are equivariantly isospectral with respect to $G$, i.e., such that there is a unitary isomorphism $U\co 
L^2(\O_1,g_1)\to L^2(\O_2,g_2)$ between the $G$-representations $\tau_1^G$ and $\tau_2^G$ (given by $\tau_i^G(g)f(x)=f(g^{-1}x)$ for 
$f\in L^2(\O_i,g_i)$, $x\in\O_i$) with the following property: $U$ maps eigenfunctions on $(\O_1,g_1)$ to eigenfunctions on 
$(\O_2,g_2)$ associated with the same eigenvalue.

Then $(\O_1/G,\bar{g}_1)$ and $(\O_2/G,\bar{g}_2)$ are isospectral orbifolds.
\end{theorem}
\begin{proof}
  Just adapt the proof of \cite[Theorem 2.7]{sutton10} to this very simple case (replacing the manifolds $M_1,M_2$ by $\O_1,\O_2$).
\end{proof}

The orbifolds from Theorem \ref{theorem:torusaction} are then seen to be equivariantly isospectral with respect to the torus $T$ from that theorem via the same argument as for the manifold version, for which the equivariant isospectrality was already observed in \cite{sutton10}.
Hence in the situation of Theorem \ref{theorem:isoorbis} with $G$ a finite subgroup of $T$ the two orbifolds $(\O/G,\bar{g}_\lambda),(\O/G,\bar{g}_{\lambda^\prime})$ are isospectral.


\begin{thebibliography}{99}

\bibitem{MR2395193}
{\sc Abreu, Miguel; Dryden, Emily B.; Freitas, Pedro; Godinho, Leonor}. 
{Hearing the weights of weighted projective planes}.
{\em Ann. Global Anal. Geom.} {\bf 33} (2008), no. 4, 373--395. 
\mrev{2395193} (2009c:58045), 
\zbl{1140.58011},
\doi{10.1007/s10455-007-9092-6},
\arx{math/0608462v1}. 

\bibitem{MR2359514}
{\sc Adem, Alejandro; Leida, Johann; Ruan, Yongbin}. 
{Orbifolds and stringy topology}. 
Cambridge Tracts in Mathematics, 171.
{\em Cambridge University Press, Cambridge}, 2007. 
xii+149 pp. ISBN: 
0-521-87004-6.
\mrev{2359514} (2009a:57044), 
\zbl{1157.57001}. 

\bibitem{MR2355369}
{\sc Bagaev, Andrey Vladimirovich.; Zhukova, Nina Ivanovna}. 
{The isometry groups of Riemannian orbifolds}.
{\em Siberian Math. J.} {\bf 48} (2007), no. 4, 579--592. 
\mrev{2355369} (2009h:53065), 
\zbl{1164.53351},
\doi{10.1007/s11202-007-0060-y}. 

\bibitem{MR1152950}
{\sc B{\'e}rard, Pierre}. 
{Transplantation et isospectralit\'e. I.}
{\em Math. Ann.} {\bf 292} (1992), no. 3, 547--559. 
\mrev{1152950} (93a:58168), 
\zbl{0735.58008}.
\doi{10.1007/BF01444635}. 

\bibitem{MR861271}
{\sc B{\'e}rard, Pierre H.} 
{Spectral geometry: direct and inverse problems}.
With appendixes by G\'erard Besson, and by B\'erard and Marcel Berger. 
Lecture Notes in Mathematics, 1207. 
{\em Springer-Verlag, Berlin}, 1986. 
xiv+272 pp. ISBN: 3-540-16788-9.
\mrev{861271} (88f:58146), 
\zbl{0608.58001},
\doi{10.1007/BFb0076330}.

\bibitem{MR1950941}
{\sc Chen, Weimin; Ruan, Yongbin}. 
{Orbifold Gromov-Witten theory}. 
{\em Orbifolds in mathematics and physics (Madison, WI, 2001)}, 
25--85, Contemp. Math., 310, 
{\em Amer. Math. Soc., Providence, RI}, 2002. 
\mrev{1950941} (2004k:53145),
\zbl{1091.53058},
\arx{math.AG/0103156v1}. 

\bibitem{MR1089240}
{\sc Chiang, Yuan-Jen}. 
{Harmonic maps of V-manifolds}.
{\em Ann. Global Anal. Geom.} {\bf 8} (1990), no. 3, 315--344. 
\mrev{1089240} (92c:58021), 
\zbl{0679.58014},
\doi{10.1007/BF00127941}. 

\bibitem{MR537804}
{\sc Donnelly, Harold}. 
{Asymptotic expansions for the compact quotients of properly discontinuous group actions}. 
{\em Illinois J. Math.} {\bf 23} (1979), no. 3, 485--496. 
\mrev{537804} (80h:58049), 
\zbl{0411.53033}.

\bibitem{MR2433665}
{\sc Dryden, Emily B.; Gordon, Carolyn S.; Greenwald, Sarah J.; Webb, David L.}
{Asymptotic expansion of the heat kernel for orbifolds}. 
{\em Michigan Math. J.} {\bf 56} (2008), no. 1, 205--238. 
\mrev{2433665} (2009h:58057), 
\zbl{1175.58010}.
\doi{10.1307/mmj/1213972406}. 
\arx{0805.3148v1}. 

\bibitem{MR2494312}
{\sc Dryden, Emily B.; Strohmaier, Alexander}. 
{Huber's theorem for hyperbolic orbisurfaces}. 
{\em Canad. Math. Bull.} {\bf 52} (2009), no. 1, 66--71. 
\mrev{2494312} (2009m:58073), 
\zbl{1179.58014},
\doi{10.4153/CMB-2009-008-0}, 
\arx{math/0504571v2}. 

\bibitem{MR1821378}
{\sc Farsi, Carla}. 
{Orbifold spectral theory}. 
{\em Rocky Mountain J. Math.} {\bf 31} (2001), no. 1, 215--235. 
\mrev{1821378} (2001k:58060), 
\zbl{0977.58025},
\doi{10.1216/rmjm/1008959678}. 

\bibitem{MR2203544}
{\sc Gilkey, Peter; Kim, Hong-Jong; Park, JeongHyeong}. 
{Eigenforms of the Laplacian for Riemannian V-submersions}.
{\em Tohoku Math. J. (2)} {\bf 57} (2005), no. 4, 505--519. 
\mrev{2203544} (2006j:58046), 
\zbl{1106.58022},
\arx{math/0310439v1}. 

\bibitem{MR1181812}
{\sc Gordon, Carolyn S., Webb, David L.; Wolpert, Scott A.} 
{Isospectral plane domains and surfaces via Riemannian orbifolds}. 
{\em Invent. Math.} {\bf 110} (1992), no. 1, 1--22. 
\mrev{1181812} (93h:58172), 
\zbl{0778.58068},
\doi{10.1007/BF01231320}. 

\bibitem{MR2044174}
{\sc Gordon, Carolyn S.; Rossetti, Juan Pablo}. 
{Boundary volume and length spectra of Riemannian manifolds: what the middle degree Hodge spectrum doesn't reveal}. 
{\em Ann. Inst. Fourier (Grenoble)} {\bf 53} (2003), no. 7, 2297--2314. 
\mrev{2044174} (2005e:58055), 
\zbl{1049.58033},
\arx{math/0111016v2}. 

\bibitem{MR1298201}
{\sc Gordon, Carolyn S.} 
{Isospectral closed Riemannian manifolds which are not locally isometric. II.}
{\em Geometry of the spectrum (Seattle, WA, 1993)}, 121--131. 
Contemp. Math., 173. 
{\em Amer. Math. Soc., Providence, RI}, 1994. 
\mrev{1298201} (95k:58166), 
\zbl{0811.58063}.

\bibitem{MR1736857}
{\sc Gordon, Carolyn S.} 
{Survey of isospectral manifolds}.
{\em Handbook of differential geometry, Vol. I.}, 747--778.
{\em North-Holland, Amsterdam}, 2000.
\mrev{1736857} (2000m:58057), 
\zbl{0959.58039}.
\doi{10.1016/S1874-5741(00)80009-6}.

\bibitem{MR2529473}
{\sc Guillemin, Victor; Uribe, Alejandro; Wang, Zuoqin}. 
{Geodesics on weighted projective spaces}.
{\em Ann. Global Anal. Geom.} {\bf 36} (2009), no. 2, 205--220. 
\mrev{2529473} (2011a:58058), 
\zbl{1176.58014},
\doi{10.1007/s10455-009-9159-7},
\arx{0805.1003v1}. 

\bibitem{herbrich11}
{\sc Herbrich, Peter}. 
{On inaudible properties of broken drums --- isospectral domains with mixed boundary conditions}. 
\arx{1111.6789v2}.

\bibitem{MR597742}
{\sc Ikeda, Akira}. 
{On lens spaces which are isospectral but not isometric}. 
{\em Ann. Sci. \'Ecole Norm. Sup. (4)} {\bf 13} (1980), no. 3, 303--315. 
\mrev{597742} (83a:58091), 
\zbl{0451.58037}.

\bibitem{MR0355886}
{\sc Kobayashi, Shoshichi}. 
{Transformation groups in differential geometry}. 
Ergebnisse der Mathematik und ihrer Grenzgebiete, Band 70. 
{\em Springer-Verlag, New York-Heidelberg}, 1972. 
viii+182 pp. ISBN: 3540058486.
\mrev{0355886} (50 \#8360), 
\zbl{0246.53031}.

\bibitem{MR1861302}
{\sc Miatello, Roberto J.; Rossetti, Juan Pablo}. 
{Flat manifolds isospectral on p-forms}. 
{\em J. Geom. Anal.} {\bf 11} (2001), no. 4, 649--667. 
\mrev{1861302} (2003d:58053), 
\zbl{1040.58014},
\doi{10.1007/BF02930761},
\arx{math/0303276v1}. 

\bibitem{MR2012261}
{\sc Moerdijk, Ieke; Mr{\v{c}}un, Janez}. 
{Introduction to foliations and Lie groupoids.}
Cambridge Studies in Advanced Mathematics, 91. 
{\em Cambridge University Press, Cambridge}, 2003. 
x+173 pp. ISBN: 0-521-83197-0.
\mrev{2012261} (2005c:58039), 
\zbl{1029.58012},
\doi{10.1017/CBO9780511615450}.

\bibitem{MR932463}
{\sc Molino, Pierre}. 
{Riemannian foliations}. 
Progress in Mathematics, 73. 
{\em Birkh\"auser Boston Inc., Boston, MA}, 1988. 
xii+339 pp. ISBN: 0-8176-3370-7.
\mrev{932463} (89b:53054), 
\zbl{0633.53001}.

\bibitem{MR0200865}
{\sc O'Neill, Barrett}. 
{The fundamental equations of a submersion}. 
{\em Michigan Math. J.} {\bf 13} (1966) 459--469. 
\mrev{0200865} (34 \#751), 
\zbl{0145.18602},
\doi{10.1307/mmj/1028999604}. 

\bibitem{MR1396675}
{\sc Pesce, Hubert}. 
Repr\'esentations relativement \'equivalentes et vari\'et\'es riemanniennes isospectrales. 
{\em Comment. Math. Helv.} {\bf 71} (1996), no. 2, 243--268. 
\mrev{1396675} (97h:58170), 
\zbl{0871.58012},
\doi{10.1007/BF02566419}. 

\bibitem{MR2761691}
{\sc Proctor, Emily; Stanhope, Elizabeth}. 
{An isospectral deformation on an infranil-orbifold}. 
{\em Canad. Math. Bull.} {\bf 53} (2010), no. 4, 684--689. 
\mrev{2761691} (2012a:58060), 
\zbl{1204.58028},
\doi{10.4153/CMB-2010-074-8},
\arx{0811.0794v1}. 

\bibitem{MR2447904}
{\sc Rossetti, Juan Pablo; Schueth, Dorothee; Weilandt, Martin}. 
{Isospectral orbifolds with different maximal isotropy orders}.
{\em Ann. Global Anal. Geom.} {\bf 34} (2008), no. 4, 351--366. 
\mrev{2447904} (2009f:58055), 
\zbl{1166.58016},
\doi{10.1007/s10455-008-9110-3},
\arx{0710.2432v2}. 

\bibitem{rueckriemen06}
{\sc R\"uckriemen, Ralf}. 
{Isospectral metrics on projective spaces}.
Thesis, Humboldt-Universit\"at zu Berlin, 2006. 
\arx{1104.2221v1}.

\bibitem{MR0079769}
{\sc Satake, Ichir\^o}. 
{On a generalization of the notion of manifold}. 
{\em Proc. Nat. Acad. Sci. U.S.A.} {\bf 42} (1956), 359--363. 
\mrev{0079769} (18,144a), 
\zbl{0074.18103}.

\bibitem{MR1895349}
{\sc Schueth, Dorothee}. 
{Isospectral metrics on five-dimensional spheres}. 
{\em J. Differential Geom.} {\bf 58} (2001), no. 1, 87--111. 
\mrev{1895349} (2003h:58052), 
\zbl{1038.58042},
\arx{math/0102126v2}.

\bibitem{MR2263484}
{\sc Shams, Naveed; Stanhope, Elizabeth; Webb, David L.}
{One cannot hear orbifold isotropy type}. 
{\em Arch. Math. (Basel)} {\bf 87} (2006), no. 4, 375--384. 
\mrev{2263484} (2007g:58039), 
\zbl{1111.58029}.
\doi{10.1007/s00013-006-1748-0}. 

\bibitem{MR2802590}
{\sc Shams Ul Bari, Naveed}. 
{Orbifold lens spaces that are isospectral but not isometric}. 
{\em Osaka J. Math.} {\bf 48} (2011), no. 1, 1--40. 
\mrev{2802590}, 
\zbl{1229.58027},
\arx{0902.2441v2}. 

\bibitem{MR2155380}
{\sc Stanhope, Elizabeth}. 
{Spectral bounds on orbifold isotropy}. 
{\em Ann. Global Anal. Geom.} {\bf 27} (2005), no. 4, 355--375. 
\mrev{2155380} (2006g:58059), 
\zbl{1085.58026},
\doi{10.1007/s10455-005-1584-7},
\arx{math/0301357v1}. 

\bibitem{MR782558}
{\sc Sunada, Toshikazu}. 
{Riemannian coverings and isospectral manifolds}. 
{\em Ann. of Math. (2)} 
{\bf 121} (1985), no. 1, 169--186. 
\mrev{782558} (86h:58141), 
\zbl{0585.58047},
\doi{10.2307/1971195}. 

\bibitem{sutton10}
{\sc Sutton, Craig}. 
{Equivariant isospectrality and Sunada's method}. 
{\em Arch. Math. (Basel)} {\bf 95} (2010), no. 1, 75--85. 
\mrev{2671240} (2011f:58060), 
\zbl{1202.53040},
\doi{10.1007/s00013-010-0139-8},
\arx{math/0608557v2}. 

\bibitem{thurston}
{\sc Thurston, William}. 
{The geometry and topology of three-manifolds (Chapter 13)}. 
\url{http://msri.org/publications/books/gt3m/}, 1978--1981.

\bibitem{weilandt07}
{\sc Weilandt, Martin}. 
{Isospectral orbifolds with different isotropy orders}. Diplom
Thesis, Humboldt-Universit\"at zu Berlin, 2007. 
urn: \href{http://nbn-resolving.de/urn:nbn:de:kobv:11-100175649}{nbn:de:kobv:11-100175649}.

\bibitem{weilandt10}
{\sc Weilandt, Martin}. 
{Isospectral metrics on weighted projective spaces}. Ph.D.
Thesis, Humboldt-Universit\"at zu Berlin, 2010. 
urn: \href{http://nbn-resolving.de/urn:nbn:de:kobv:11-100175726}{nbn:de:kobv:11-100175726}.

\end{thebibliography}
\end{document}